\documentclass[sn-mathphys-num]{sn-jnl}

\usepackage{graphicx}%
\usepackage{multirow}%
\usepackage{amsmath,amssymb,amsfonts}%
\usepackage{amsthm}%
\usepackage{mathrsfs}%
\usepackage[title]{appendix}%
\usepackage{xcolor}%
\usepackage{textcomp}%
\usepackage{manyfoot}%
\usepackage{booktabs}%
\usepackage{algorithm}%
\usepackage{algorithmicx}%
\usepackage{algpseudocode}%
\usepackage{listings}%

\newcommand{\eq}{\mathbin{<\!>}}

\theoremstyle{thmstyleone}%
\newtheorem{theorem}{Theorem}
\newtheorem{proposition}[theorem]{Proposition}%
\newtheorem{corollary}[theorem]{Corollary}%

\newtheorem{lemma}[theorem]{Lemma}%

\theoremstyle{thmstyletwo}%

\theoremstyle{thmstylethree}%
\newtheorem{definition}[theorem]{Definition}%

\newtheorem{remark}[theorem]{Remark}

\newtheorem{example}[theorem]{Example}

\begin{document}

\title[Infinite-dimensional Convex Cones]{Infinite-dimensional Convex Cones: Internal Geometric Structure and Analytical Representation}


\author*{\fnm{Valentin V.} \sur{Gorokhovik}}\email{gorokh@im.bas-net.by}

\affil*{\orgdiv{Institute of Mathematics}, \orgname{The National Academy of Sciences of Belarus}, \orgaddress{\street{Surganov, 11}, \city{Minsk}, \postcode{220012}, \country{Belarus}}}


\abstract{In the paper we consider convex cones in infinite-dimensional real vector spaces which are endowed with no topology. The main purpose is to study an internal geometric structure of convex cones and to obtain an analytical description of those. To this end, we first introduce the notion of an open component of a convex cone and then prove that an arbitrary convex cone is the disjoint union of the partial ordered family of its open components and, moreover, as an ordered set this family is an upper semilattice. We identify the structure of this upper semilattice with the internal geometric structure of a convex cone. We demonstrate that the internal geometric structure of a convex cone is related to its facial structure but in the infinite-dimensional setting these two structures may differ each other. Further, we study the internal geometric structure of conical halfspaces (convex cones whose complements are also convex cones).  We show that every conical halfspace is the disjoint union of the linear ordered family of its open components each of which is a conical halfspace in its linear hull. Using the internal geometric structure of conical halfspaces, each asymmetric conical halfspace is associated with a linearly ordered family of linear functions, which generates in turn a real-valued function, called a step-linear one, analytically describing this conical halfspace. At last, we establish that an arbitrary asymmetric convex cone admits an analytical representation by the family of step-linear functions.
}

\keywords{Convex cone, Conical halfspace, Relative algebraic interior, Step-linear function, Analytical representation}


\pacs[MSC Classification]{52A05,  06F20}

\maketitle

\section{Introduction}\label{sec1}

Cones, especially convex cones, play an important role in variational analysis and optimization. It is enough to recall tangent and normal cones to sets \cite{RW,Mord}, the recession and barrier cones to convex sets \cite{KhTZ,Roc}, cones of positive vectors in ordered vector spaces \cite{Pere,Jameson}, vector optimization problems \cite{Gor-m,Jahn}, linear and second-order conic optimization problems \cite{BonSha} and much more.
Besides the publications listed above convex cones have been intensively studied in a vast literature that cover both finite-dimensional and infinite-dimensional settings, see, e.g., \cite{Alip_Tour,Barker,Khazayel,NZ} and the references therein.
The paper \cite{Khazayel} is devoted to algebraic properties of convex sets and convex cones in the infinite-dimensional setting. Along with new results the paper \cite{Khazayel} contains a rather detailed overview of recent papers devoted to the same topic. The authors of the papers \cite{NZ} have studied the interrelation between the relatively algebraic interior (also called the intrinsic core) of a convex cone and its positive dual cone.

In the present paper we consider convex cones in infinite-dimensional real vector spaces which are endowed with no topology. The main purpose is to study the geometric structure of convex cones and to obtain an analytical description of those. The geometric structure of convex cones as well as arbitrary convex sets is most often related to their facial structure. The facial structure of closed pointed convex cones  in the finite-dimensional vector spaces has been exhaustively studied in the paper \cite{Barker}. It was shown in \cite{Barker} that in the finite-dimensional setting the family of all faces of a closed convex cone, partially ordered by inclusion, is a complete lattice. The main results of the paper \cite{Barker} characterize algebraic properties of the lattice of faces. In the recent paper \cite{Millan} the authors clarify the relation between intrinsic core and facial structure of convex sets in general vector spaces, and, in particular, the relation between points of intrinsic cores and minimal faces containing of such points.

In Section 2 of the paper we propose another approach to the geometric structure of convex cones. As the basic elements of the geometric structure we consider not faces but so called `open components' of a convex cone. In order to define the concept of an open component of a convex cone, we introduce on this cone the preorder relation, called the dominance relation, which generates in turn an equivalence relation on the cone in question. Open components of a convex cone are equivalence classes of the factor of this cone by the equivalence relation generated by the dominance relation. It is established that each class equivalence is an relatively algebraic open convex subcone of the cone in question. This justifies the term `open component'. The family of all open components of a convex set is ordered by the partial order which is defined as a factor of the dominance relation. Thus, we see that an arbitrary convex cone is the disjoint union of the partial ordered family of its open components and, moreover, we prove that as an ordered set this family is an upper semilattice. It allows us to consider open components as disjoint `building blocks' of a convex cone, and, moreover, in the cone these `building blocks' arrange relatively each other not irregularly but in accordance with the structure of the upper semilattice corresponding to the cone in question. We refer to such description of the geometric structure of a convex cone as its internal geometric structure.

The relations between open components and faces of a convex cone are discussed in Section \ref{sec3}. It is shown that each open component is the intrinsic core of the minimal face containing this open component. Moreover, the family of open components one-to-one corresponds to the subfamily of those faces whose intrinsic cores are nonempty.  Thus, the claim that every convex cone is the disjoint union of its open components can be reformulated as follows: each convex cone is the disjoint union of the nonempty intrinsic cores of its faces. In such form
this fact was pointed out earlier in the papers \cite{Millan,Garcia}. We demonstrate also (see Theorem \ref{th3.13b}, Corollary \ref{cor3.13a} and Example \ref{ex3.17}) that in each infinite-dimensional space there are such convex cones faces of which may have empty intrinsic cores. Thus, in the infinite-dimensional setting the internal geometric structure of a convex cone may differ from its facial structure.

Further, in Section \ref{sec4} we study the internal geometric structure of conical halfspaces (convex cones whose complements are also convex cones).  We show that every conical halfspace is the disjoint union of the linear ordered family of its open components each of which is a conical halfspace in its linear hull.  Section \ref{sec5} is devoted to an analytical representation of conical halfspaces. Using the internal geometric structure of conical halfspaces, each asymmetric conical halfspace is associated with a linearly ordered family of linear functions, which generates in turn a real-valued function, called a step-linear one, analytically describing this conical halfspace. At last, in Section \ref{sec6} we establish that an arbitrary asymmetric convex cone admits an analytical representation by the family of step-linear functions.

We conclude this introductory section by recalling some preliminary notions and results related to convex cones.

Throughout this paper $X$ is a nontrivial ($X \ne \{0\}$) real vector space. 

A subset $K \subseteq X$ is called a \textit{cone} if $\lambda x \in K$ for all $x \in K$ and all $\lambda >0$.

In other words a subset $K \subseteq X$ is a cone if along with each point $x \in K$ the subset $K$ contains the entire ray emanating from the origin and passing through the point $x$.

Note, that the origin not necessarily belongs to a cone $K$.

Recall that set $Q \subseteq X$ is said to be {\it convex} if for any $x,y \in Q$ the segment $[x,y]:=\{\alpha x + (1-\alpha)y \mid \alpha \in [0,1]\}$ is entirely contained in $Q$.

A cone $K \subseteq X$ is convex if and only if $x + y \in K$ whenever $x,y \in K$ or, equivalently, if and only if $K + K \subseteq K$.

The linear hull of a convex cone $K$ denoted by ${\rm Lin}K$ is equal to $K-K$.

A cone $K \subseteq X$ is said to be asymmetric if $K \bigcap (-K) = \varnothing$. A convex cone $K \subset X$ is asymmetric if and only if $0 \not\in K$.

Let $K \subseteq X$ be a convex cone. The set
$$
L_K:=\{h \in X \mid x +th \in K \,\,\text{for all}\,\, x \in K\,\,\text{and all}\,\, t \in {\mathbb R}\}
$$
is the greatest vector subspace in $X$ such that $L_K + K=K$.

We refer to $L_K$ as the \textit{vector subspace associated with the convex cone $K$}.

Note also, that $0 \in L_K$ for any convex cone $K$ and hence $L_K$  is nonempty for any convex cone $K$, with $K \bigcup L_K$ being a convex cone as well. Moreover, when $K \bigcap (-K) =\varnothing$, i.e. when  a convex cone $K$ is asymmetric, $L_K$ is the greatest vector subspace among all vector subspaces $T \subset X$ such that $K \bigcup T$ is a convex cone. A convex cone $K \subset X$ is asymmetric if and only if  $K \bigcap L_K = \varnothing$. It is easily seen that $L_K = K \bigcap (-K)$ provided $K \bigcap (-K) \ne \varnothing$.

Clearly $L_K = L_{-K}$ for any convex cone $K \subseteq X$. Furthermore, for any convex cone $K \subseteq X$ the inclusion $L_{K} \subset {\rm Lin}(K)$, where ${\rm Lin}(K)$ is the linear hull of $K$, is true. Indeed, for every $x \in K$ we have $x+L_K \subset K$ and, consequently, $L_K = (x + L_K) - x \subset K-K = {\rm Lin}(K)$.

For any convex cone $K \subset X$ the inclusion $K\bigcap (-K) \subseteq L_{\widehat{K}}$, where  $\widehat{K} := K \setminus (-K)$ is the asymmetric part of $K$, holds. The next example shows that this inclusion can be proper.
\begin{example}
{\rm Let $K := \{(x_1,x_2) \in {\mathbb{R}}^2 \mid x_2 > 0 \} \bigcup \{(0,0)\}$. Then $\widehat{K} = K \setminus \{(0,0)\}$ and $L_{\widehat{K}} = \{(x_1,x_2) \in {\mathbb{R}}^2 \mid x_2 = 0 )\}$ while $L_K = \{(0,0)\}$.}
\end{example}

An important class of convex cones is formed by conical half-spaces.

 A subset $H$ of a real vector space $X$ is called \cite{Las,GS98,GS2000} a \textit{halfspace in $X$} if both $H$ and its complement $X \setminus H$ are convex. A halfspace $H$ in $X$ which is a cone is called a \textit{conical halfspace in $X$}. For any halfspace $H$ in $X$ its complement $X  \setminus H$ is a halfspace in $X$ too, and, if $H$ is a conical halfspace in $X$, then $X \setminus H$ is also a conical halfspace in $X$, and either $H$ or $X\setminus H$ is asymmetric.

\begin{lemma}{\rm \cite{GS98,GS2000}}\label{lem1.1}
An asymmetric convex cone $H \subset X$ is a conical halfspace in $X$ if and only if the set $X \setminus ((-H)\bigcup H)$ is a vector subspace in $X$. Moreover, for an asymmetric conical halfspace $H$ in $X$ the equality $L_H = X \setminus ((-H)\bigcup H)$ holds and hence $X = (-H) \bigcup L_H \bigcup H$ with $(-H) \bigcap H = \varnothing$ and $H \bigcap L_H = \varnothing$. If a convex cone $H \subset X$ is not asymmetric, i.e., if $0 \in H$, then $H$ is a conical halfspace in $X$ if and only if $X = (-H) \bigcup H$.
\end{lemma}

It follows from the statements of Lemma \ref{lem1.1} that ${\rm Lin}(H) = X$ for any conical halfspace $H \subset X$.

A widespread class of pairs of complementary halfspaces in $X$ is given by the subsets
$$
H_>(l, \alpha):=\{x \in X \mid l(x) > \alpha\}\,\,\text{and}\,\,H_\le(l, \alpha):=\{x \in X \mid l(x) \le \alpha\},
$$
where $l:X \to {\mathbb{R}}$ is non-zero linear function on $X$ and $\alpha \in {\mathbb{R}}$ is an arbitrary real.

Semispaces and their complements give other examples of complementary halfspaces in $X$. Recall \cite{Hammer,Klee} that a subset $S \subset X$ is called a \textit{semispace generated by a point $a \in X$} (or, simply, a \textit{semispace at a point $a \in X$}) if it is a maximal (by inclusion) convex set in $X \setminus \{a\}$. Semispaces at the origin are asymmetric conical halfspaces in $X$.

Semispaces were introduced and studied by Hammer \cite{Hammer} and Klee \cite{Klee}.  Halfspaces (also called hemispaces) in finite-dimensional spaces were studied by Lassak \cite{Las} and Martinez-Legaz and Singer \cite{ML_S88} and in infinite-dimensional spaces by Lassak and Pruszynski \cite{LasPro}, Gorokhovik and Semenkova \cite{GS98}, and Gorokhovik and Shinkevich \cite{GS2000}. It is worth noting that in the English translation of the paper \cite{GS98}, halfspaces are mistakenly translated as semispaces.

In conclusion of this section we note that in what follows some statements will be proved only for asymmetric convex cones. Using the representation of an arbitrary convex cone $K$ as $K = [K\bigcap (-K)]\bigsqcup \widehat{K}$ ($\bigsqcup$ denotes a disjoint union), where $\widehat{K}:= K \setminus (-K)$ is the asymmetric part of the cone $K$, most of the statements proved for asymmetric convex cones can be extended to arbitrary convex cones.

\section{Internal geometric structure of convex cones}


To characterize the internal geometric structure of convex cones we will need the notions of the algebraic interior and the relative algebraic interior of convex sets (see \cite{Millan,Khazayel,Klee51,Raikov,Rubinstein1,Rubinstein,Holmes,Aliprantis,KhTZ,Zalinescu}).

The \textit{algebraic interior} (or the \textit{core}) of a subset $Q \subset X$, denoted by ${\rm cr}Q$,  is the subset of $Q$ consisting of all points $x \in Q$ such that the intersection of $Q$ with each straight line of $X$ passing through $x$ contains an open interval around $x$.

The \textit{relative algebraic interior} (or the \textit{intrinsic core}) of a subset $Q \subset X$, denoted by ${\rm icr}Q$, is the subset of $Q$ consisting of all points $x \in Q$ such that the intersection of $Q$ with each straight line laying in ${\rm aff}Q$ and passing through $x$ contains an open interval around $x$.

More precisely, $x \in {\rm cr}Q$ if and only if for any $y \in X \setminus \{x\}$  there exists a positive real number $\delta > 0$ such that $x+t(y-x) \in Q$ for all $t \in (-\delta,\delta)$, while $x \in {\rm icr}Q$ if and only if for any $y \in {\rm aff}Q \setminus \{x\}$  there exists a positive real number $\delta > 0$ such that $x+t(y-x) \in Q$ for all $t \in (-\delta,\delta)$.

If a set $Q \subset X$ is convex then $x \in {\rm icr}Q$ if and only if for any $y \in Q \setminus \{x\}$  there exists a positive real number $\delta > 0$ such that $x+t(y-x) \in Q$ for all $t \in (-\delta,\delta)$.

 A set $Q \subset X$ is called \textit{(relatively) algebraic open} if it coincides with its (relative) algebraic interior, i.e., if $Q = {\rm cr}Q$ ($Q = {\rm icr}Q$).

\begin{proposition}\label{pr2.1}
Let $K$ be a convex cone from $X$. Then
\begin{equation}\label{e2.1}
x \in {\rm icr}K \Longleftrightarrow x \in K\,\,\text{and}\,\,(\forall y \in K) (\exists \lambda >0)\,\,\text{such that}\,\, x - \lambda y \in K.
\end{equation}
\end{proposition}
\begin{proof}
Suppose that ${\rm icr}K \ne \varnothing$. Since $K$ is convex, $x \in {\rm icr}K$ if and only if for any $y \in K$ there exists a real $\delta > 0$ such that $x + t (y-x) = (1-t)x +ty \in K$ for all $t \in (-\delta, \delta)$. Considering that $K$ is a cone and that for any $t \in (-\delta,0)$ the real $1-t = 1+|t|$ is positive we obtain $x - \displaystyle\frac{|t|}{1+|t|}y \in K$ for all $t \in (-\delta,0)$. Since $\displaystyle\frac{|t|}{1+|t|} > 0$ for any $t \in (-\delta,0)$ and the choice of $y \in K$ was arbitrary, the implication ($\Longrightarrow$) is proved.

To prove the implication ($\Longleftarrow$) we suppose that a vector $x \in K$ is such that for any $y \in K$ there exists $\lambda >0$  such that $x - \lambda y \in K$.
Since $(\lambda - t)y \in K$ for all $t \in (0,\lambda)$, by convexity of the cone $K$ we obtain $(x-\lambda y) +(\lambda -t)y = x -ty \in K$ for all $t \in (0,\lambda)$. Besides, $tx \in K$ for all $t > 0$. Using convexity $K$ again we get $x - t(y-x) \in K$ for all $t \in (0,\lambda)$.
At last, since both $x$ and $y$ belong to the convex cone $K$ we have $x + t(y-x) \in K$ for all $t \in [0,1]$. Hence, $x + t(y-x) \in K$ for all $t \in (-\delta,\delta)$, where $\delta = \min\{\lambda,1\}$. This proves that $x \in {\rm icr}K$.
\end{proof}

It is well known that when $X$ is finite-dimensional, the intrinsic core of every convex subset $Q \subset X$ is nonempty. The next example shows that in each infinite-dimensional real vector space there exist convex sets (moreover, convex cones) whose intrinsic cores are empty.

\begin{example}{\rm \cite[Ch.2, $\S$ 7]{Raikov}}\label{ex2.1}
{\rm Let $X$ be a real vector space and let $\{e_i, i \in I\}$ be a Hamel basis for $X$. Let us define the convex cone $K$ in $X$ which consists of non-zero vectors $x \in X$ whose components $\{x_i,i \in I\}$ in the given basis $\{e_i, i \in I \}$ is nonnegative. Since the basis $\{e_i, i \in I\}$ is infinite, each vector $x \in K$ has zero components. Suppose that $x_j =0$. Since $e_j \in K$ and for any real $\lambda > 0$ the vector $x - \lambda e_j$ does not belong to $K$, then $x \notin {\rm icr}K$ and, hence, through the equivalence \eqref{e2.1}, ${\rm icr}K = \varnothing$ as $x$ is an arbitrary vector from $K$.}
\end{example}

\subsection{The dominance relation on a convex cone}

The equivalence \eqref{e2.1} can be interpreted as follows: a point $x \in K$ belongs to ${\rm icr}K$ if and only if it dominates in a certain sense all other points of the cone $K$. Restricting the right part of \eqref{e2.1} to pairs of points $\{x,y\}$ from $K$ we get the following binary relation defined on $K$.

\begin{definition}
    {\rm We say that a point $x \in K$ \textit{dominates} a point $y \in K$, and denote this by $y \unlhd_K x$, if there is a positive real $\lambda > 0$ such that $x - \lambda y \in K$.}
\end{definition}

\smallskip

It is not difficult to verify that $\unlhd_K$ is a preorder relation (i.e., a reflexive and transitive binary relation) on $K$. By the symbol $\eq_K$ we denote the symmetric part of $\unlhd_K$ defined by $y_1 \eq_K y_2 \Leftrightarrow y_1 \unlhd_K y_2, y_2 \unlhd_K y_1$. It is easy to see that $y_1 \eq_K y_2$ if and only if there exist reals $\mu > 0$ and $\nu > 0$ such that $y_1 - \mu y_2 \in K$ and $ y_2 - \nu y_1 \in K$. The asymmetric part of $\unlhd_K$ denoted by $\lhd_K$ is defined by $y_1 \lhd_K y_2 \Leftrightarrow y_1 \unlhd_K y_2, y_2 \hspace{5pt}/\hspace{-10pt}\unlhd_K y_1$. The relation $\lhd_K$ is a strict partial order, while  $\eq_K$ is an equivalence relation.

\smallskip

\begin{remark}
{\rm The dominance relation $\unlhd_K$ on a convex cone $K$ is related to the relation introduced by Rubinstein in \cite[p.~187]{Rubinstein} for studying the facial structure of convex sets (see also \cite{Klee}).}
\end{remark}

\smallskip

Using the dominance relation $\unlhd_K$ we can reformulate the assertion of Proposition~\ref{pr2.1} as follows: for any convex cone $K \subset X$ the following equivalence holds:
\begin{equation}\label{e2.22}
x \in {\rm icr}K \Longleftrightarrow x \in K\,\,\text{and}\,\,y \unlhd_K x\,\,\text{for all}\,\,y \in K,
\end{equation}
with $y \lhd_K x$ for all $y \in K \setminus {\rm icr}K$ and all $ x \in {\rm icr}K$.

In addition, if a convex cone $K \subset X$ is not asymmetric, then $L_K = K\bigcap (-K)$ and it directly follows from the definition of $\unlhd_K$ and $L_K$ that
\begin{equation}\label{e2.23}
x \in L_K \Longleftrightarrow x \in K\,\,\text{and}\,\,x \unlhd_K y\,\,\text{for all}\,\,y \in K,
\end{equation}
with $x \lhd_K y$ for all $x \in L_K$ and all $y \in K\setminus L_K$.

It immediately follows from the properties of a convex cone $K$ and the definition of the dominance relation $\unlhd_K$ that for any $x,y \in K$ the implications
\begin{equation}\label{e3.4}
x - y \in K \Longrightarrow y \unlhd_K x\,\,\,\,\text{and}\,\,\,\,x -y \in L_K \Longrightarrow x \eq_K y
\end{equation}
hold.

Moreover, for any $y,y_1,y_2,x \in K$ and any positive reals $\alpha > 0, \beta > 0$ the following implications are also true:
\begin{equation}\label{e3.5}
y \unlhd_K x \Longrightarrow \alpha y \unlhd_K \beta x;\,\,\,\,\text{}\,\,\,\,y_1 \unlhd_K x, y_2 \unlhd_K x \Longrightarrow y_1+y_2 \unlhd_K x;.
\end{equation}
\begin{equation}\label{e3.5a}
x \unlhd_K y_1, x \unlhd_K y_2 \Longrightarrow x \unlhd_K y_1+y_2.
\end{equation}

For any $x \in K$ we define the sets
$$
F_K(x) :=\{y \in K \mid y \unlhd_K x\}\,\,\,\,\text{and}\,\,\,\,\widehat{F}_K(x) :=\{y \in K \mid y \lhd_K x\}.
$$
Since $x \in F_K(x)$ the set $F_K(x) \ne \varnothing$ for any $x \in K$. Moreover, it follows from \eqref{e3.5} that for every $x \in K$ the set $F_K(x)$ is a convex subcone of the cone $K$. As for the set $\widehat{F}_K(x)$ it can be empty for some $x \in K$. In the case when
$\widehat{F}_K(x) \ne \varnothing$ it is also a cone which is, in general, not necessarily convex.

\begin{example}\label{ex3.5}
{\rm Let $K = \{(x_1, x_2) \mid x_1 \ge 0, x_2 \ge 0\} \setminus \{(0,0)\}.$
Then $\widehat{F}_K(x)$ is empty for any $x =(x_1, x_2) \in K$ such that $x_1 > 0, x_2 =0$ or $x_1 = 0, x_2 >0$. For $x =(x_1, x_2) \in K$ with $x_1 > 0, x_2 >0$ the set $\widehat{F}_K(x)= \{(x_1, x_2) \mid x_1 > 0, x_2 =0\}\cup \{(x_1, x_2) \mid x_1 = 0, x_2 >0\}$ is a cone which is not convex.}
\end{example}

It follows from transitivity of the relation $\unlhd_K$ that $F_K(y) \subseteq F_K(x)$ whenever $y \unlhd_K x$. Evidently, the converse is also true. Thus, for any $x,y \in K$ the following equivalences are true:
\begin{equation}\label{e3.6a}
y \unlhd x \Leftrightarrow F_K(y) \subseteq F_K(x); y \lhd_K x \Leftrightarrow F_K(y) \subsetneq F_K(x); y \eq_K x \Leftrightarrow F_K(y) = F_K(x).
\end{equation}

The set $G_K(x):= \{y \in K \mid x \unlhd_K y\}$ is also a nonempty convex cone for every $x \in K$ (this follows from the implications \eqref{e3.5} and \eqref{e3.5a}) and, like the set $F_K(x)$, it can be used for characterization of $\unlhd_K$ on $K$, since for any $x,y \in K$ the following equivalences
\begin{equation*}
y \unlhd_K x \Leftrightarrow G_K(y) \supseteq G_K(x); y \lhd_K x \Leftrightarrow G_K(y) \supsetneq G_K(x); y \eq_K x \Leftrightarrow G_K(y) = G_K(x)
\end{equation*}
hold.

\subsection{Open components of a convex cone}

For a convex cone $K \subset X$ by $\mathcal{O}(K):= K/\eq_K$ we denote the set of equivalence classes of $K$ corresponding to the equivalence relation $\eq_K$.

Clearly, that the cone $K$ is a disjoint union of all equivalence classes from $\mathcal{O}(K)$, i.e. $K = \bigcup\{E\mid E\in \mathcal{O}(K)\}$ and $E_1\bigcap E_2 = \varnothing$ for any $E_1,E_2 \in \mathcal{O}(K),E_1 \ne E_2$.
In addition, $\mathcal{O}(K)$ is assumed to be equipped with the partial order $\unlhd_K^*$ which is defined as follows: $E_1 \unlhd_K^* E_2$ holds for $E_1,E_2 \in \mathcal{O}(K)$ if and only if $x_1 \unlhd_K x_2$ for all (some) $x_1 \in E_1$ and all (some) $x_2 \in E_2$.

The equivalence class containing the vector $x \in K$ will be denoted by $E_x$. 

\begin{proposition}\label{pr3.6}
Each equivalence class $E$ from $\mathcal{O}(K)$ is a relatively algebraic open convex cone and, moreover, $E = {\rm icr}F_K(E)$ where $F_K(E):=\bigcup\{\tilde{E} \in \mathcal{O}(K) \mid \tilde{E} \unlhd_K^* E\}$.
\end{proposition}

\begin{proof}
Let $E \in {\mathcal{O}}(K)$ and let $x \in E$. Then $E = E_x := \{y \in K \mid y \eq_K x\}$. It is easily seen that $E_x = F_K(x) \bigcap G_K(x)$, where $F_K(x) := \{y \in K \mid y \unlhd_K x\}$ and
$G_K(x):= \{y \in K \mid x \unlhd y\}$. We have seen above that $F_K(x)$ and $G_K(x)$ are convex cones and, consequently, $E_x$ as the intersection of two convex cones is a convex cone too.

Note also that for every $x \in E$ the equality $F_K(E) = F_K(x)$ holds. Let $y \in {\rm icr}F_K(x)$ for some $x \in E$. Since $F_K(x)$ is a convex cone, it follows from Proposition~\ref{pr2.1} that for each $z \in F_K(x)$ there exists $\lambda_z > 0$ such that $y - \lambda_z z \in F_K(x)$. Hence, for $x \in F_K(x)$ there exists $\lambda_x > 0$ such that $y - \lambda_x x \in F_K(x) \subset K$. This implies that $x \unlhd_K y$. In addition, since $y \in F_K(x)$, we have that $y \unlhd_K x$. Thus, $y \eq_K x$, i.e., $y \in E_x$, and we have proved that ${\rm icr}F_K(x) \subset E_x$.

To prove the reverse inclusion we observe that $y \unlhd_K x$ for all $y \in F_K(x)$. From this, applying \eqref{e2.22}, we conclude that $x \in {\rm icr}F_K(x)$ and consequently  $E_x \subseteq {\rm icr}F_K(x)$.  Indeed, let $z \in E_x$. Then $x \eq_K z$ and, since $y \unlhd_K x$ for all $y \in F_K(x)$, by transitivity of $\unlhd_K$ we obtain that $y \unlhd_K z$ for all $y \in F_K(x)$. Due to \eqref{e2.22} this implies that $z \in {\rm icr}F_K(x)$.
\end{proof}

\begin{definition}
{\rm Elements of $\mathcal{O}(K)$ will be called \textit{relatively algebraic open components} (or, shortly, \textit{open components}) of the convex cone $K$.}
\end{definition}
\smallskip

\begin{proposition}\label{pr2.23}
For the intrinsic core ${\rm icr}K$ of a convex cone $K$  to be non-empty, it is necessary and sufficient that the family $\mathcal{O}(K)$ of open components of the cone $K$ ordered by the partial order $\unlhd_K^*$ have the greatest element, i.e., such element $E_{\text{sup}} \in \mathcal{O}(K)$ that $E \unlhd^* E_{\text{sup}}$ for all $E \in \mathcal{O}(K)$. Moreover, in this case, ${\rm icr}K = E_{\text{sup}}$.
\end{proposition}

\smallskip

The assertion of this proposition is immediate from the equivalence \eqref{e2.22} and from that $E_{\text{sup}}$ is relatively algebraic open.

\smallskip

\begin{example}[Revisiting Example \ref{ex2.1}]
{\rm Let $X$ be a infinite-dimensional real vector space and let $\{e_i,i \in I\}$ be a Hamel basis for $X$. Consider a convex cone $K \subset X$ consisting of non-zero vectors $x \in X$ whose coordinates $\{x_i,i\in I\}$ in the given basis $\{e_i,i \in I\}$ are nonnegative. For $x \in K$ by $I(x)$ we denote the subset of those indices from $I$ for which $x_i > 0$, that is, $I(x) := \{i \in I \mid x_i > 0\}$. The subset $I(x)$ is non-empty and finite for every $x \in K$. It is not difficult to see that $y \unlhd_K x \Leftrightarrow I(y) \subseteq I(x)$ and $y \eq_K x \Leftrightarrow I(y) = I(x)$. We conclude from this that $E$ is an open component of a cone $K$, $E \in {\mathcal{O}}(K)$, if and only if $E = \{ x \in K \mid I(x) = J\}$, where $J$ is a some finite subset of $I$. Furthermore, since each finite subset of $I$ is a proper subset of another finite subset of $I$ the family of open components of the cone $K$ does not have the greatest element and hence, through Proposition \ref{pr2.23}, ${\rm icr}K = \varnothing$. The latter property of the cone $K$ under consideration was directly proved previously in Example \ref{ex2.1}.}
\end{example}

\smallskip

\begin{corollary}
A convex cone $K \subset X$ is relatively algebraic open if and only if the family $\mathcal{O}(K)$ of its open components is a singleton, i.e., if and only if $\mathcal{O}(K) = \{K\} =\{{\rm icr}K\}$.
\end{corollary}

When $0 \in K$ it is easily verified that $E_0 = K \bigcap (-K) = L_K$, and $E_0$ is the least element of the partial ordered set $(\mathcal{O}(K), \unlhd_K^*)$, that is $E_0 \unlhd_K^* E$ for all $E \in \mathcal{O}(K)$. The next theorem shows that, regardless of whether $0 \in K$ or $0 \not \in K$, the partial ordered set $(\mathcal{O}(K), \unlhd_K^*)$ is an upper semilattice.

Recall \cite{Birkhoff} that a partially ordered set $(Z,\preceq)$ is called an \textit{upper semilattice} if each pair of elements $\{z_1,z_2\}
 \subset Z$ has the least upper bound in $Z$ commonly denoted by $z_1 \vee z_2$.

\begin{theorem}\label{th3.6}
The set $\mathcal{O}(K)$ of all open components of a convex cone $K$ equipped with the partial order $\unlhd_K^*$ is an upper semilattice and, in addition, the equality $E_x \vee E_y = E_{x+y}$ holds for all $x,y \in K$.
\end{theorem}

\begin{proof}
Let $E_1,E_2 \in \mathcal{O}(K)$ and let $x \in E_1$ and $y \in E_2$, i.~e., $E_1 =E_x, E_2=E_y$. Since $(1 -\lambda) x \in K$ for all $\lambda \in (0,1)$ and $y \in K$, then $ (x+y) - \lambda x \in K$ for all $\lambda \in (0,1)$ and consequently $x \unlhd_K x+y.$ Similarly, we get $y \unlhd_K x+y$. Hence $E_x \unlhd_K^* E_{x+y}$ and $E_y \unlhd_K^* E_{x+y}$.

Suppose now that $E_x \unlhd_K^* E$ and $E_y \unlhd_K^* E$ for some $E \in \mathcal{O}(K)$. Let $u \in E$. Then $x \unlhd_K u$ and $y \unlhd_K u$ or equivalently $u - \lambda x \in K$ and $u - \mu y \in K$ for some $\lambda >0, \mu >0$. Suppose that $\lambda \le \mu$. Then, since $K$ is a convex cone, we obtain that $u - \displaystyle\frac{\lambda}2 (x+y) \in K$. It proves that $x+y \unlhd_K u$ and hence $E_{x+y} \unlhd_K^* E$. Thus, $E_{x+y}=E_x \vee E_y$.
\end{proof}

Due to the equality $K = \bigsqcup \{E \mid E \in {\mathcal O}(K)\}$ we can consider open components from ${\mathcal O}(K)$ as `building blocks' for the cone $K$, and, moreover, in $K$ these blocks are arranged relatively to each other not irregularly but in accordance with the structure of the upper semilattice $({\mathcal O}(K), \unlhd^*)$.

\smallskip

This observation justifies the following definition.

\begin{definition} {\rm The structure of the upper semilattice $(\mathcal{O}(K), \unlhd_K^*)$ generated by a convex cone $K$ will be referred to as the \textit{internal geometric structure of the cone $K$}}.
\end{definition}

\smallskip

We will conclude this section with two propositions describing some subcones of a convex cone $K$ in terms of open components from ${\mathcal O}(K)$.

\begin{proposition}\label{pr3.7}
Let $E \in \mathcal{O}(K)$ be an arbitrary open component of a convex cone~$K$. Then ${\rm Lin}(E)\bigcap K$ is equal to the convex cone $F_K(E):=\bigcup\{\tilde{E} \in {\mathcal O}(K) \mid \tilde{E} \unlhd^*_K E\}$, i.e. ${\rm Lin}(E)\bigcap K = F_K(E)$.
\end{proposition}

\begin{proof}
First of all, we observe that, since $E$ is a convex cone, ${\rm Lin}(E)= E-E$.

Let $\tilde{E},E \in \mathcal{O}(K)$ be such that $\tilde{E} \unlhd_K^* E$. Since $\tilde{E} \vee E =E$, then for any $y \in \tilde{E}$ and any $z \in E$ we have $x:= y + z \in E.$ Consequently, $y=x-z \in E-E={\rm Lin}(E)$. It proves that $F_K(E) \subset {\rm Lin}(E)\bigcap K$.

Conversely, let $y \in {\rm Lin}(E)\bigcap K$. It follows from $y \in {\rm Lin}(E)$ that $y = x - z$ with $x,z \in E$. Since $x =y + z$, then $E_y\vee E= E$ and, consequently, $E_y \unlhd_K^* E$. It proves the reverse inclusion ${\rm Lin}(E)\bigcap K \subset F_K(E)$.
\end{proof}

\begin{proposition}
Assume that a subfamily $\mathcal{E}$ of the family of open components $\mathcal{O}(K)$ of a convex cone $K$ is linearly ordered by the relation $\unlhd_K^*$. Then for any $E \in \mathcal{E}$ which is not the least element of ${\mathcal{E}}$  the set $\widehat{F}_{\mathcal{E}}(E):= \bigcup\{\tilde{E} \in \mathcal{E} \mid \tilde{E} \lhd_K^* E\}$ is also a  convex cone and ${\rm Lin}(\widehat{F}_{\mathcal{E}}(E))\bigcap E = \varnothing$, where $\tilde{E} \lhd_K^* E \Leftrightarrow \tilde{E} \unlhd_K^* E, \tilde{E} \ne E$.
\end{proposition}

\begin{proof}
The positive homogeneity of $\widehat{F}_{\mathcal{E}}(E)$ follows from the positive homogeneity of each $E \in \mathcal{E}$. In addition, since the subfamily $\mathcal{E}$ is linearly ordered by $\unlhd_K^*$, then for any $y,z \in \widehat{F}_{\mathcal{E}}(E)$ we have either $E_{y+z} = E_y$ or $E_{y+z} = E_z$. Hence, $E_{y+z} \lhd_K^* E$ and $y+z \in \widehat{F}_{\mathcal{E}}(E)$. Thus, $\widehat{F}_{\mathcal{E}}(E)$ is a convex cone.

Assume now that ${\rm Lin}(\widehat{F}_{\mathcal{E}}(E))\bigcap E \ne \varnothing$ and let $z \in {\rm Lin}(\widehat{F}_{\mathcal{E}}(E))\cap E$. Since $\widehat{F}_{\mathcal{E}}(E))$ is a convex cone, ${\rm Lin}(\widehat{F}_{\mathcal{E}}(E)) = \widehat{F}_{\mathcal{E}}(E) - \widehat{F}_{\mathcal{E}}(E)$ and, consequently,  $z = u - v$ with $u,v \in \widehat{F}_{\mathcal{E}}(E)$. It implies that $z+v=u \in \widehat{F}_{\mathcal{E}}(E)$. On the other hand, since $v \lhd_K z$, it follows from Theorem \ref{th3.6} that $z+v \in E_z=E$. We arrive at a contradiction with $E\bigcap \widehat{F}_{\mathcal{E}}(E) = \varnothing$. This proves that ${\rm Lin}(\widehat{F}_{\mathcal{E}}(E))\bigcap E = \varnothing$.
\end{proof}

Simple examples show that when the subfamily $\mathcal{E} \subset \mathcal{O}(K)$ is not linearly ordered by $\unlhd_K^*$, the cone $\widehat{F}_{\mathcal{E}}(E)$ can be non-convex and $E \in \mathcal{E}$ can belong to ${\rm Lin}(\widehat{F}_{\mathcal{E}}(E))$.

\begin{example}[cf. with Example \ref{ex3.5}]
{\rm Let $X={\mathbb{R}}^2$ and $K:= {\mathbb{R}}^2_+ \setminus \{(0,0)\}$. The set $\mathcal{O}(K)$ of open components of $K$ consists of three elements: $E_1:= \{(y_1,y_2)\mid y_1 >0,y_2=0\}$, $E_2:=\{(y_1,y_2)\mid y_1 =0,y_2>0\}$, and $E_{1,2}:=\{(y_1,y_2)\mid y_1 >0,y_2>0\}$. The subfamily $\mathcal{E} = \mathcal{O}(K)$ is not linearly ordered by $\unlhd_K^*$ and $\widehat{F}_{\mathcal{E}}(E_{1,2})=\bigcup\{E \in \mathcal{E} \mid E \lhd_K^* E_{1,2}\} = E_1\bigcup E_2$ is non-convex cone with $E_{1,2} \subset {\rm Lin}(\widehat{F}_{\mathcal{E}}(E_{1,2})) = {\mathbb{R}}^2$.}
\end{example}

\section{Open components and faces of a convex cone}\label{sec3}

Let $Q$ be a convex set in a real vector space $X$.

A nonempty convex subset $F \subset Q$ is called a \textit{face} of $Q$ (see, for instance, \cite{Barker,KhTZ,Millan}) if it satisfies the following property: if for some $u,v \in Q$ there exists $\alpha \in (0,1)$ such that $\alpha u + (1-\alpha)v \in F$ then $u,v \in F$.

In other words, a nonempty convex subset $F \subset Q$ is \textit{a face} of $Q$ if every segment of $Q$, having in its relative interior an element of $F$, is entirely contained in $F$. The set $Q$ itself is its own face and the empty set is considered as a face of any convex set $Q$.

An intersection of any family of faces of $Q$ is also a face of $Q$ while for an union of some family of faces of $Q$ to be a face of $Q$, it is sufficient that  this family be linearly ordered by inclusion (proofs of these assertions can be found in \cite{Millan}, for example).

It is easy to see that when $K$ is a convex cone, each face $F$ of $K$ is a convex cone too. Moreover, when $K$ is an asymmetric convex cone, each face $F$ of $K$ is an asymmetric convex cone as well.

\begin{theorem}\label{th3.13}
For any open component $E \in {\mathcal O}(K)$ of a convex cone $K$ the set $F_K(E):=\bigcup\{\tilde{E} \in \mathcal{O}(K) \mid \tilde{E} \unlhd_K^* E\}$ is the minimal (by inclusion) face of $K$ containing $E$. Moreover, ${\rm icr}F_K(E) \ne \varnothing$ and ${\rm icr}F_K(E) = E$.

Conversely, for every face $F$ of convex cone $K$ such that ${\rm icr}F \ne \varnothing$ there exists an uniquely determined open component $E \in {\mathcal O}(K)$ such that $F = F_K(E)$ and hence ${\rm icr}F = E$.
\end{theorem}

\begin{proof}
Let $E \in \mathcal{O}(K)$. Since $F_K(E) = F_K(x)$ for any $x \in E$ and $F_K(x)$ is a convex cone for all $x \in K$, then $F_K(E)$ is a convex cone too (this follows also from Proposition \ref{pr3.7}). To prove that $F_K(E)$ is a face of $K$  we consider an arbitrary $x \in F_K(E)$ and assume that $x = \alpha u + (1 - \alpha)v$ for some $\alpha \in (0,1)$ and some  $u,v \in K$. Since $E_u$ and $E_v$ are convex cones then $\alpha u \in E_u$ and $(1 - \alpha)v \in E_v$ and, due to Theorem \ref{th3.6}, $x \in E_{u+v}=E_u \vee E_v$. It shows that  $E_u \vee E_v = E_x = E$ and hence $E_u \unlhd_K^* E$ and $E_v \unlhd_K^* E$. Consequently, $u \in E_u \subset F_K(E)$ and $v \in E_v \subset F_K(E)$. This proves that $F_K(E)$ is a face of $K$.

To prove that $F_K(E)$ is the minimal face of $K$ containing $E$ we need to show that $F_K(E) \subseteq F$ for any face $F$ of $K$ such that $E \subset F$.

 Take an arbitrary face $F$ of $K$ containing $E$, and let $x \in E \subset F$ and $y \in F_K(E)=F_K(x)$. Since $F_K(x) = \{u \in K \mid u \unlhd_K x\}$, by definition of $\unlhd_K$ we have that there exists $\lambda > 0$ such that $x -\lambda y \in K$. It follows from convexity of $K$ that $(1- \alpha)x+ \alpha(x - \lambda y) = x - \alpha \lambda y \in K$ for  all $\alpha \in (0,1)$. Hence we can assume that $x -\lambda y \in K$ for some $\lambda \in (0,1)$. Since $x - \lambda y = (1 - \lambda)x +\lambda(x-y) = (1- \lambda)\left(x + \displaystyle\frac{\lambda}{1-\lambda}(x-y)\right) \in K$ and $K$ is a cone we conclude that $x = \lambda y + (1-\lambda)z$, where $z := x + \displaystyle\frac{\lambda}{1-\lambda}(x-y) \in K$.

Thus, we have $x \in F$, where $F$ is a face of $K$, and $x = \lambda y + (1-\lambda)z$, where $y,z \in K$ and $\lambda \in (0,1)$. Hence, by the definition of a face, $y,z \in F$. Since $y$ is an arbitrary element of $F_K(x) = F_K(E)$ we conclude that $F_K(E) \subseteq F$.

The condition ${\rm icr}F_K(E) \ne \varnothing$ and the equality ${\rm icr}F_K(E) = E$ were actually proved in Proposition \ref{pr3.6}.

It remains to prove the second assertion of the theorem.

Let $F$ be a face of $K$ such that ${\rm icr}F \ne \varnothing$ and let $x \in {\rm icr}F$. Since $F$ is a convex cone, by \eqref{e2.22} we have that $y \unlhd_K x$ for all $y \in F$. It implies that $F \subseteq F_K(x)$. As it was proved above $F_K(x)$ is the minimal (by inclusion) face of $K$ containing $x$, hence $F=F_K(x)=F_K(E_x)$ and consequently ${\rm icr}F = E_x$. The latter equality shows that $E_x$ does not depend on the choice $x \in {\rm icr}F$. In addition, since for each $x \in K$ there exists an only open component $E \in {\mathcal O}(K)$ containing $x$, the open component $E$ satisfying the equalities $F=F_K(E)$ and ${\rm icr}F = E$ is unique too.
\end{proof}

\begin{remark}
{\rm It follows from the above theorem that the family of open components of a convex cone $K$ coincides with the family of nonempty intrinsic cores of faces of $K$, i.e. ${\mathcal O}(K) = \{{\rm icr}F \mid F\,\,\text{is a face of}\,\,K\,\,\text{such that}\,\,{\rm icr}F \ne \varnothing\}$. Thus, the assertion that every convex cone is the disjoint union of its open components can be reformulated as follows: each convex cone is the disjoint union of the nonempty intrinsic cores of its faces. In such form this fact was pointed out earlier in the papers \cite{Millan,Garcia}.}
\end{remark}

\begin{corollary}\label{cor3.19}
Let $K$ be a convex cone. For any point $x \in K$ the set $F_K(x):=\{y \in K \mid y \unlhd_K x\}$ is the minimal (by inclusion) face of $K$ containing $x$, and ${\rm icr}F_K(x) = E_x \ne \varnothing$.
\end{corollary}

\begin{remark}
{\rm Taking into account the equivalences \eqref{e3.6a} and the assertion of Corollary~\ref{cor3.19} we can define the relation $\unlhd_K$ via minimal faces containing a point. This was noted by one of the Referees.}
\end{remark}
\smallskip

To formulate the next theorem we need the following notion.

\smallskip

Let $(Z,\preceq)$ be a linearly ordered set. An ordered pair $(A,B)$ of nonempty subsets of $Z$ is said \cite{Rosenstein} to be a \textit{cut} of $(Z,\preceq)$ if $A \bigcap B = \varnothing, A \bigcup B = Z$, and $u \prec v$ for all $u \in A$ and all $v \in B$. (Here $u \prec v$ means that $u \preceq v$ but $v \not\preceq u$.)
The sets $A$ and $B$ are respectively called the \textit{lower class} and the \textit{upper class} of a cut $(A,B)$.

\smallskip

\begin{theorem}\label{th3.13b}
Assume that the family ${\mathcal O}(K)$ of open components of a convex cone $K$ is linearly ordered by the relation $\unlhd^*_K$. For any cut $\Sigma := (\underline{\mathcal O}(K),\overline{\mathcal O}(K))$ of $({\mathcal O}(K),\unlhd^*_K)$ the set $F_K(\Sigma):= \bigcup\{E \mid E \in \underline{\mathcal O}(K)\}$ is a face of $K$, and

$($i$)$ if the lower class $\underline{\mathcal O}(K)$ has the greatest element $E_{max}$, then $F_K(\Sigma)=F_K(E_{max})$ and hence ${\rm icr}F_K(\Sigma) = {\rm icr}F_K(E_{max}) = E_{max} \ne \varnothing$;

$($ii$)$ in the case when $\underline{\mathcal O}(K)$ does not have the greatest element, ${\rm icr}F_K(\Sigma) = \varnothing$.
\end{theorem}

\begin{proof}
With each  $\tilde{E} \in \underline{\mathcal{O}}(K)$  we associate the set $F_K(\tilde{E}) = \bigcup\{\hat{E} \in \mathcal{O}(K) \mid \hat{E} \unlhd_K^* \tilde{E}\}$ which is a face of $K$ due to Theorem \ref{th3.13} proved above. It is easy to see that $F_K(\Sigma) = \bigcup\{F_K(\tilde{E}) \mid \tilde{E} \in \underline{\mathcal{O}}(K)\}$ and consequently $F_K(\Sigma)$ is the union of the linearly ordered family of faces of $K$. By Proposition~2.8 from \cite{Millan} the set $F_K(\Sigma)$ is also a face of $K$.

\smallskip

The assertion (\textit{i}) is evident. To prove (\textit{ii}) we suppose to the contrary that $\underline{\mathcal O}(K)$ does not have the greatest element and ${\rm icr}F_K(\Sigma) \ne \varnothing$. By Theorem \ref{th3.13} we have, since ${\rm icr}F_K(\Sigma) \ne \varnothing$,  that there exists $\tilde{E} \in {\mathcal O}(K)$ such that $F_K(\Sigma) = F_K(\tilde{E})$. From the latter equality we conclude that $\tilde{E} \in \underline{\mathcal O}(K)$ and, since the elements of ${\mathcal O}(K)$ do not intersect each other, $E \unlhd^*_K \tilde{E}$ for all $E \in \underline{\mathcal O}(K)$. This contradicts the assumption that $\underline{\mathcal O}(K)$ does not have the greatest element.
\end{proof}

Below we will need the following corollary of Theorem \ref{th3.13b}.

\begin{corollary}\label{cor3.13a}
 Assume that the family $\mathcal{O}(K)$ of a convex cone $K$ is linearly ordered by $\unlhd_K^*$. Then for any open component $E \in {\mathcal O}(K)$ which is not the least element of ${\mathcal O}(K)$ the set $\widehat{F}_K(E):=\bigcup\{\tilde{E} \in \mathcal{O}(K) \mid \tilde{E} \lhd_K^* E\}$ is  a face of $K$, and

\smallskip

$($i$)$ if $\{\tilde{E} \in {\mathcal O}(K)\mid \tilde{E} \lhd^* E\}$ has the greatest element $E_{max}$, then $\widehat{F}_K(E)=F_K(E_{max})$ and hence ${\rm icr}\widehat{F}_K(E) = {\rm icr}F_K(E_{max}) = E_{max} \ne \varnothing$;

\smallskip

$($ii$)$ in the case when $\{\tilde{E} \in {\mathcal O}(K)\mid \tilde{E} \lhd^* E\}$ does not have the greatest element, ${\rm icr}\widehat{F}_K(E)= \varnothing$.
\end{corollary}

For proving of this corollary it is enough to note that the pair ($\{\tilde{E} \in {\mathcal O}(K)\mid \tilde{E} \lhd^* E\}, \{\tilde{E} \in {\mathcal O}(K)\mid E \unlhd^* \tilde{E}\}$) is a cut of $({\mathcal O}(K),\unlhd^*_K)$.

\smallskip

When a convex cone $K$ is such that the family ${\mathcal O}(K)$ is finite, the internal geometric structure of $K$, i.e. the upper semilattice $({\mathcal O}(K), \unlhd^*_K)$, is order isomorphic to the family of faces of $K$ partially ordered by the set-theoretic inclusion. However, when ${\mathcal O}(K)$ is infinite the internal geometric structure of a convex cone $K$ may differ from its facial structure.

\smallskip

\begin{example}\label{ex3.17}
{\rm Let $X$ be an infinite-dimensional real vector space and $B = \{e_i\}_{i \in I}$ be a Hamel basis of $X$. Since $X$ is infinite-dimensional we can choose a countable subset $B_0 \subset B$. Take one element from $B_0$ and denote it by $e_{+\infty}$. Then, enumerate elements from $B_0 \setminus \{e_{+\infty}\}$ by integer numbers, so we have $B_0 = \{e_i,i \in {\mathbb Z},e_{+\infty}\}$, where ${\mathbb Z}$ stands for the set of integer numbers. For each $x \in X$ by $\{x_i, i \in {\mathbb Z},x_{+\infty}\}$ we denote its coordinates in the basis $B$ that correspond to basis vectors from $B_0$. Since each vector $x \in X$ have only a finite number of nonzero coordinates in the basis $B$, the sum $x_{+\infty} + \sum_{i \in {\mathbb Z}(m)}x_i$, where ${\mathbb Z}(m)= \{i \in {\mathbb Z} \mid i \le m\}$, is well-defined for any $x \in X$ and any integer number $m$. Define the following sets $E_{+\infty} := \{x \in X \mid l_{+\infty}
(x) > 0\}$ and $E_m := \{x \in X \mid l_{+\infty}(x) = 0; l_s(x)=0, s \in {\mathbb Z}, m < s; l_m(x) > 0\}\,\,\text{for}\,\,m \in {\mathbb Z},$ where $l_{+\infty}(x) := x_{+\infty} + \sum_{i \in {\mathbb Z}}x_i$ and $l_s(x) := \sum_{i \in {\mathbb Z}(s)}x_i$ for $s \in {\mathbb Z}$.

Since  $l_{+\infty}$ and $l_s, s \in {\mathbb Z},$ are linear functions on $X$, the sets $E_{+\infty}$ and $E_m, m \in {\mathbb Z},$ are convex cones in $X$. Moreover, if $m < n$ for $n \in {\mathbb Z}\cup \{+\infty\}$ and $m \in {\mathbb Z}$ then for every $x \in E_n$ and $y \in E_m$ we have $x+y \in E_n$. Indeed, if $y \in E_m$ then $l_{+\infty}(y) = 0; l_s(y)=0, s \in {\mathbb Z}, m < s$ and if $x \in E_n$ then $l_{+\infty}(x) = 0; l_s(x)=0, s \in {\mathbb Z}, n < s$ and $l_n(x) > 0$. Therefore, for $x + y$ we have $l_{+\infty}(x+y) = 0; l_s(x+y)=0, s \in {\mathbb Z}, n < s$ and $l_n(x+y) > 0$. Hence, $x + y \in E_n$.

It follows from these properties that the set $K := (\bigcup_{m \in {\mathbb Z}}E_m)\cup E_{+\infty}$ is an asymmetric convex cone (moreover, $K$ is an asymmetric conical halfspace in $X$) with $L_K = \{x \in X \mid l_{+\infty}(x) = 0; l_m(x)=0, m \in {\mathbb Z}\}$.

The relation $\unlhd_K$ is defined on $K$ as follows: $y \lhd_K x$ if and only if ($x \in E_{+\infty}$ and $y \in \bigcup_{m \in {\mathbb Z}}E_m$) or ($x \in E_m$ and $y \in E_n$ with $n,m \in {\mathbb Z}$ such that $n < m$); $x \eq_K y$ if and only if $x,y \in E_m$ for some $m \in {\mathbb Z} \cup \{+\infty\}$. We conclude that $({\mathcal O}(K),\unlhd_K^*)$ is such that ${\mathcal O}(K) = \{E_m,m \in {\mathbb Z},E_{+\infty}\}$ and $\unlhd_K^*$ is the linear order which is identical with the natural ordering of ${\mathbb Z}\cup \{+\infty\}$.

The open component $E_{+\infty}$ is the greatest element of $({\mathcal O}(K),\unlhd_K^*)$, while the subfamily $\{E \in {\mathcal O}(K) \mid E \lhd_K^* E_{+\infty}\} = \{E_m, m \in {\mathbb Z}\}$ does not have the greatest element. Thus, due to the assertion ($ii$) of Corollary \ref{cor3.13a} $\widehat{F}_K(E_{+\infty}) = \bigcup_{m \in {\mathbb Z}}E_m$ is a face of $K$ but for it there is no $E \in {\mathcal O}(K)$ such that $\widehat{F}_K(E_{+\infty}) = F_K(E)$.}
\end{example}

\smallskip

Theorem \ref{th3.13} shows that the mapping $E \mapsto F_K(E)$ which maps the upper semilattice $({\mathcal O}(K), \unlhd^*_K)$ of open components of a convex cone $K$ into the lattice of all faces of the same convex cone $K$ partially ordered by inclusion is injective and preserves order relations. However, this mapping is not surjective in general, since its full image  is only the subfamily of those faces whose intrinsic core is nonempty.  At the same time,  as it follows from Theorem \ref{th3.13b} and Example \ref{ex3.17}, in the case  when ${\mathcal O}(K)$ is infinite,  $K$ may have a lot of faces with empty intrinsic cores. Hence, the internal geometric structure of a convex cone may not be order isomorphic its facial structure.

\section{Internal geometric structure of conical halfspaces}\label{sec4}

In this section we focus on a special class of convex cones, namely on conical halfspaces (the definition of conical halfspaces and some their properties were briefly discussed  in Section~\ref{sec1}).

As above, the symbol $\unlhd_H$ denotes the dominance relation on the conical halfspace $H \subset X$ and $\mathcal{O}(H)$ is the family of open components of $H$ ordered by the partial order $\unlhd_H^*$ which is the factorization of $\unlhd_H$ by $\eq_H$.

\begin{proposition}
Let $H \subset X$ be a conical halfspace in a real vector space $X$.
Then the dominance relation $\unlhd_H$ defined on $H$ is total.
\end{proposition}

\begin{proof}
Suppose that $H$ is an asymmetric  conical halfspace in $X$, then one of the following three alternatives, $y - x \in H$, $x - y \in H$, $y - x \in L_H$, holds for any $x,y \in H$.  Consequently, in this case the assertion follows from the implications \eqref{e3.4}.

Now let a conical halfspace $H$ be not asymmetric. Then $H = L_H \bigcup \widehat{H}$, where $\widehat{H}:= H \setminus L_H$ is the  asymmetric part of $H$, and moreover, since $\widehat{H} = - (X \setminus H)$, then $\widehat{H}$ is an asymmetric conical halfspace. The restriction of $\unlhd_H$ on $\widehat{H}$ coincides with $\unlhd_{\widehat{H}}$ and hence, as proven above, is total on $\widehat{H}$. In addition, through \eqref{e2.23}, we have that $x \unlhd_H y$ for all $x \in L_H$ and all $y \in H$. Since $H = L_H \bigcup \widehat{H}$, it proves that $\unlhd_H$ is total on $H$.
\end{proof}

\begin{corollary}
The family $\mathcal{O}(H)$ of open components of a conical halfspace $H$ is linearly ordered by $\unlhd_H^*$.
\end{corollary}

\begin{proof}
Let $E_1,E_2 \in \mathcal{O}(H)$ be such that $E_1 \ne E_2$. For $x \in E_1,y \in E_2$ the alternative $x - y \in L_H$ is impossible because in this case, through \eqref{e3.4}, we would have $x \eq_H y$, which contradicts $E_1 \ne E_2$. Thus, we have either $y - x \in H \setminus L_H$ or $x - y \in H \setminus L_H$ and, consequently, again through \eqref{e3.4}, either $y \unlhd_H x$ or $x \unlhd_H y$. Since $x \in E_1, y \in E_2$  and $E_1 \ne E_2$, we actually have either $y \lhd_H x$ or $x \lhd_H y$. It implies that either $E_1 \lhd_H^* E_2$ or $E_2 \lhd_H^* E_1$.
\end{proof}

\begin{proposition}\label{pr4.1a}
Let $H$ be a conical halfspace in a real vector space $X$. Then for any $x,y \in H$ the following implication
$$
y \lhd_H x \Longrightarrow x - y \in H \setminus L_H.
$$
holds.
\end{proposition}

\begin{proof}
Let $x,y \in H$ and $y \lhd_H x$.
Since $H \subset X$ is a conical halfspace, for every pair $x,y \in X$ exactly one of the following three alternatives
$x - y \in H \setminus L_H,\,\,y - x \in H \setminus L_H$ or $x-y \in L_H,$ holds.
But, through the implications \eqref{e3.4}, both alternatives $y - x \in H \setminus L_H$ and $x-y \in L_H,$
are impossible, since they imply respectively $x \unlhd_H y$ and $x \eq_H y$ which contradict the condition $y \lhd_H x$.
\end{proof}

\begin{proposition}\label{pr4.2a}
For every open component $E \in {\mathcal O}(H)$ of a conical halfspace $H \subset X$ the inclusions
$$
L_H \subset L_E \subset {\rm Lin}(E)
$$
hold
\end{proposition}

\begin{proof}
It was shown in Section \ref{sec1} that $L_K \subset {\rm Lin}(K)$ for any convex cone $K$. Thus, $L_E \subset {\rm Lin}(E)$.

Now, let $h \in L_H$. By the definition of $L_H$, we have
$$
x+th \in H\,\,\text{for all}\,\,x \in H\,\,\text{and all}\,\,t \in {\mathbb{R}}.
$$
Evidently,
$$
x+th \in H\,\,\text{for all}\,\,x \in E\,\,\text{and all}\,\,t \in {\mathbb{R}}.
$$
To prove the inclusion $L_H \subset L_E$ we need to show that actually
\begin{equation}\label{eA5}
x+th \in E\,\,\text{for all}\,\,x \in E\,\,\text{and all}\,\,t \in {\mathbb{R}}.
\end{equation}
To justify this, we suppose that \eqref{eA5} is not true. Then there exist $y \in E$ and $\lambda \in {\mathbb{R}}$ such that $y+\lambda h  \in \tilde{E}$ where $\tilde{E} \in {\mathcal O}(H),\,\tilde{E} \ne E$. Since ${\mathcal O}(H)$ is linearly ordered by $\unlhd_H^*$, we have either $\tilde{E} \lhd_H^* E$ or $E \lhd_H^* \tilde{E}$. Assume that $\tilde{E} \lhd_H^* E$, then $y+{\lambda}h \lhd_H y$ and, through Proposition~\ref{pr4.1a},  $y - (y +{\lambda}h)= -{\lambda}h \in H \setminus L_H$. But  this contradicts $h \in L_H$. In the case $E \lhd_H^* \tilde{E}$ we would have ${y} \lhd_H {y}+{\lambda}h$ and then ${y} + {\lambda}h - y = \lambda h \in H \setminus L_H$, that again contradicts $h \in L_H$. Thus, any $h \in L_H$ satisfies \eqref{eA5} and, hence, $L_H \subset L_E$.
\end{proof}

In what follows we will mainly deal with asymmetric conical halfspaces, but this does not restrict the generality. Indeed, each non-asymmetric conical halfspace $H$ can be presented as $H= L_H \bigsqcup \widehat{H}$, where $\widehat{H} :=H \setminus (-H)$ is the asymmetric part of $H$, which is an asymmetric conical halfspace, and $L_{\widehat{H}} = L_H$. Moreover, the restriction of $\unlhd_H$ on $\widehat{H}$ coincides with $\unlhd_{\widehat{H}}$ and, since $y \lhd_H x$ for all $y \in L_H$ and all $x \in \widehat{H}$, the equality ${\mathcal O}(\widehat{H}) = {\mathcal O}(H) \setminus \{L_H\}$ holds. We see from this that any assertion related to open components of an asymmetric conical halfspace can be extended to open components of a conical halfspace which is not asymmetric.

\smallskip

For the proof of the next theorem we need the following counterpart of Lemma~\ref{lem1.1}.

Let $Y$ be a vector subspace in a vector space $X$. We say that a subset $H \subset Y$ is  a \textit{conical halfspace in $Y$} if both $H$ and $Y \setminus H$ are convex cones.

\smallskip

\begin{lemma}\label{lem4.1}
An asymmetric convex cone $H \subset Y$ is a conical halfspace in $Y$ if and only if the set $Y \setminus (H\bigcup(-H))$ is a vector subspace. Moreover, the vector subspace $L_H$ associated with  an asymmetric conical halfspace $H$ in $Y$ coincides with the vector subspace $Y \setminus (H\bigcup(-H))$.
\end{lemma}

\smallskip

\begin{theorem}
Let $H \subset X$ be an asymmetric conical halfspace in a real vector space $X$ and let ${\mathcal{O}}(H)$ be the family of open components of $H$.
Then the following assertions hold:

$($i$)$ for each open component $E \in {\mathcal{O}}(H)$ the minimal face of $H$ containing $E$, that is the set $F_H(E)$, is a conical halfspace in ${\rm Lin}(E)$ and $L_{F_H(E)} = L_H$;

$($ii$)$ for each open component $E \in {\mathcal{O}}(H)$ which is not the least element of ${\mathcal{O}}(H)$ the face $\widehat{F}_H(E)$ corresponding to $E$ is a conical halfspace in the vector subspace ${\rm Lin}(\widehat{F}_H(E))$ and $L_{\widehat{F}_H(E))} = L_H$;

$($iii$)$ each open component $E \in {\mathcal{O}}(H )$ is a conical halfspaces in the vector subspace ${\rm Lin}(E)$ and
\begin{equation}\label{e4.11}
L_E = {\rm Lin}(\widehat{F}_H(E)) = (-\widehat{F}_H(E))\cup L_H \cup \widehat{F}_H(E);
\end{equation}

\end{theorem}

\begin{proof}
($i$) First, we recall that due to Theorem \ref{th3.13} for each $E \in {\mathcal O}(H)$ the set  $F_H(E)$ is the minimal face of $H$ containing $E$ and hence $F_H(E)$ is an asymmetric convex cone. Further, since $H$ is an asymmetric conical halfspace, we have $X = (-H)\bigsqcup L_H \bigsqcup H$. Consequently, ${\rm Lin}(E) = {\rm Lin}(E)\bigcap X =({\rm Lin}(E)\bigcap (-H))\bigsqcup ({\rm Lin}(E)\bigcap L_H) \bigsqcup {\rm Lin}(E)\bigcap H)$. Propositions \ref{pr3.7} and \ref{pr4.2a} give us respectively that ${\rm Lin}(E)\bigcap H = F_H(E)$ and ${\rm Lin}(E)\bigcap L_H =L_H$ and, consequently,
\begin{equation}\label{e4.10a}
{\rm Lin}(E) = (-F_H(E))\bigsqcup L_H \bigsqcup F_H(E).
\end{equation}
Due to Lemma \ref{lem4.1} this equality implies that $F_H(E)$ is an asymmetric conical halfspace in the vector subspace ${\rm Lin}(E)$ and $L_{F_H(E)} =L_H$.

($ii$) Since $E$ is not the least element of the linearly ordered set $({\mathcal{O}}(H), \unlhd_H^*)$, the family $\{\tilde{E} \in {\mathcal{O}}(H) \mid \tilde{E} \lhd_H^* E\}$ is nonempty and, consequently, the cone $\widehat{F}_H(E)$ is also nonempty.
Note that the nonempty family of vector subspaces $\{{\rm Lin}(\tilde{E}) \mid \tilde{E} \in {\mathcal{O}}(H), \tilde{E} \lhd_H^* E\}$ is linearly ordered by inclusion and hence $\bigcup\{{\rm Lin}(\tilde{E}) \mid \tilde{E} \in {\mathcal{O}}(H), \tilde{E} \lhd_H^* E\}$ is also a vector subspace. Furthermore, since $\tilde{E} \subset \widehat{F}_H(E)$  and consequently ${\rm Lin}(\tilde{E}) \subset {\rm Lin}(\widehat{F}_H(E))$ for every $\tilde{E} \in {\mathcal{O}}(H), \tilde{E} \lhd_H^* E$, we have $\bigcup \{{\rm Lin}(\tilde{E})\mid \tilde{E} \in {\mathcal{O}}(H)\mid \tilde{E} \lhd_H^* E\} \subset {\rm Lin}(\widehat{F}_H(E))$. On the other hand, the converse inclusion follows from $\widehat{F}_H(E) \subset \bigcup \{{\rm Lin}(\tilde{E})\mid \tilde{E} \in {\mathcal{O}}(H)\mid \tilde{E} \lhd^* E\}$ and thus we have
\begin{equation}\label{e4.10b}
{\rm Lin}(\widehat{F}_H(E)) = \bigcup \{{\rm Lin}(\tilde{E})\mid \tilde{E} \in {\mathcal{O}}(H), \tilde{E} \lhd_H^* E\}.
\end{equation}
Further, from the assertion ($i$) proved above, using the equality \eqref{e4.10a}, we obtain ${\rm Lin}(\widehat{F}_H(E)) = \bigcup \{(-F_H(\tilde{E}))\bigsqcup L_H \bigsqcup F_H(\tilde{E})\mid \tilde{E} \in {\mathcal{O}}(H), \tilde{E} \lhd_H^* E\} = (-\widehat{F}_H(E))\bigsqcup L_H \bigsqcup \widehat{F}_H(E)$ and hence, due to Lemma \ref{lem4.1}, $\widehat{F}_H(E)$ is an asymmetric conical halfspace in the vector subspace ${\rm Lin}(\widehat{F}_H(E))$, and $L_{\widehat{F}_H(E)} = L_H$.

($iii$) Since $F_H(E) = E\bigsqcup \widehat{F}_H(E)$, it immediately follows from the equalities \eqref{e4.10a} and \eqref{e4.10b} that
 ${\rm Lin}(E) = (-E)\bigsqcup {\rm Lin}(\widehat{F}_H(E)) \bigsqcup E$. The latter equality shows, due to Lemma~\ref{lem4.1}, that $E$ is a conical halfspace in the vector space ${\rm Lin}(E)$, and $L_E = {\rm Lin}(\widehat{F}_H(E))$.
\end{proof}

\begin{remark}
{\rm It follows from \eqref{e4.11} and \eqref{e4.10b} that
\begin{equation}\label{e4.17}
L_E = \bigcup\{{\rm Lin}(\tilde{E}) \mid \tilde{E} \in {\mathcal O}(H), \tilde{E} \lhd_H^*E\}
\end{equation}
for each $E \in {\mathcal O}(H)$.

}
\end{remark}

\smallskip

We conclude this section with the following theorem.

\begin{theorem}
Let $H \subset X$ be a conical halfspace, and let ${\mathcal O}(H)$ be the family of open components of $H$ linearly ordered by the relation $\unlhd^*_H$. Then
the family of vector subspaces $\{{\rm Lin}(E) \mid E \in {\mathcal{O}}(H)\}$ is linearly ordered by inclusion and is order isomorphic to $({\mathcal{O}}(H),\unlhd_H^*)$; the latter means that  for any $E_1,E_2 \in {\mathcal{O}}(H)$ the inclusion ${\rm Lin}(E_1) \subsetneq {\rm Lin}(E_2)$ holds if and only if $E_1 \lhd_H^* E_2$. In addition,
\begin{equation}\label{e4.17a}
X =\bigcup \{{\rm Lin}(E)\mid E \in {\mathcal{O}}(H)\}.
\end{equation}
\end{theorem}

\begin{proof}
It follows immediately from the definition of the set $F_H(E)$ that $F_H(E_1) \subsetneq F_H(E_2)$ holds for any $E_1,E_2 \in {\mathcal{O}}(H)$ such that $E_1 \lhd_H^* E_2$. Thus, the assertion we need to prove is an obvious consequence of the equality \eqref{e4.10a}.

To prove the equality \eqref{e4.17} we observe that each $x \in X$ belongs either to $L_H$ or to $F_H(E)\cup (-F_H(E))$ for some $E \in {\mathcal{O}}(H)$.
\end{proof}

\begin{remark}
{\rm As it was noted in Example \ref{ex3.17} the convex cone which was considered in that example is in fact an asymmetric conical halfspace. It is worth observing that this example shows that in each infinite-dimensional vector space there exist conical spaces, whose internal geometric structure differs from their facial structure.
}
\end{remark}

\section{Analytical representation of asymmetric conical halfspaces}\label{sec5}

Let $X$ be a real vector space, $H$ be a conical halfspace in $X$, and $({\mathcal O}(H),\unlhd_H^*)$ be the family of open components of $H$.

We begin with an analytical representation of open components \linebreak $E \in ({\mathcal O}(H),\unlhd_K^*)$ and the vector subspaces $L_E$ and ${\rm Lin}(E)$ generated by $E$.

\begin{proposition}\label{pr4.11}
Let $H \subset X$ be an asymmetric conical halfspace in $X$. For each open component $E \in {\mathcal{O}}(H)$ there is a nonzero linear function $l_E: X \to {\mathbb{R}}$ such that
\begin{equation}\label{e4.18}
E =\{x \in {\rm Lin}(E) \mid l_E(x) > 0\}\,\,\text{and}\,\,
L_E = \{x \in {\rm Lin}(E) \mid l_E(x) = 0\}.
\end{equation}
Furthermore, $l_E(x) = 0$ for all $x \in {\rm Lin}(\tilde{E})$ provided that $\tilde{E} \in {\mathcal O}(H), \tilde{E} \lhd_H^* E$.
\end{proposition}

\begin{proof}
Note that $E$ and $L_E$ are convex subsets in ${\rm Lin}(E)$ and $E \cap L_E = \varnothing$. Furthermore, $E$ is relatively algebraic open (see Proposition~\ref{pr3.6}).  By Corollary 5.62 from \cite{Aliprantis} there is a nonzero linear functional $l_E$ defined on $X$ which properly separates $E$ and $L_E$. It is not hard to verify that $l_E$ satisfies the equalities \eqref{e4.18}.

The last assertion follows from the second equality in \eqref{e4.18} and the equality \eqref{e4.17}.
\end{proof}

Thus, every asymmetric conical halfspace $H \subset X$ can be associated with the following family ${\mathcal F}_H$ of nonzero linear functions:
$$
{\mathcal F}_H := \{l_E \in {\mathcal L}(X,{\mathbb{R}}) \mid E \in {\mathcal O}(H), l_E\,\,\text{satisfies}\,\,\eqref{e4.18}\},
$$
where ${\mathcal L}(X,{\mathbb{R}})$ stands for the vector space of linear functions defined on $X$.

Notice that the family ${\mathcal F}_H$ is non uniquely  defined.

We endow
the family ${\mathcal F}_H$ with the linear order $\unlhd'_H$ defined as follows:
$$
l_E \unlhd'_H l_{\tilde{E}} \Longleftrightarrow {\tilde{E}} \unlhd_H^* E.
$$
Clearly $({\mathcal F}_\prec, \unlhd'_H)$ is order anti-isomorphic to $({\mathcal O}(H),\unlhd_H^*)$.

\begin{proposition}
The family ${\mathcal F}_H$ is linearly independent.
\end{proposition}

\begin{proof}
Arguing by contradiction, assume that there is a finite subfamily $\{l_{E_1},l_{E_2},\ldots,l_{E_k}\} \subset {\mathcal F}_H$ and nonzero real numbers $\alpha_1,\alpha_2,\ldots,\alpha_k$ such that $\alpha_1l_{E_1}(x)+\alpha_2l_{E_2}(x)+\ldots+\alpha_kl_{E_k}(x)=0$ for all $x \in X$. Without loss of generality we may assume that $l_{E_1} \lhd'_Hl_{E_2} \lhd'_H \ldots \lhd'_H l_{E_k}$. Then $E_k \lhd_H^* E_{k-1} \lhd_H^*\ldots \lhd_H^* E_2 \lhd_H^* E_1$ and hence, through the second assertion of Proposition \ref{pr4.11}, for every $i = 1,2,\ldots,k-1$ the equality $l_{E_i}(x)=0$ holds for all $x \in {\rm Lin}(E_k)$. Since $-\alpha_kl_{E_k}(x)=\alpha_1l_{E_1}(x)+\alpha_2l_{E_2}(x)+\ldots+\alpha_{k-1}l_{E_{k-1}}(x)$ for all $x \in X$, we have $l_{E_k}(x)=0$ for all $x \in {\rm Lin}(E_k)$, but this contradicts $l_{E_k}(x)>0$ for all $x \in E_k$.
\end{proof}

\begin{theorem}\label{th4.161}
Let $H \subset X$ be an asymmetric conical halfspace in $X$ and let ${\mathcal F}_H = \{l_E \mid E \in {\mathcal O}(H)\}$  be a family of linear functions associated with $H$.
For any open component ${E} \in {\mathcal O}(H)$, which is not the greatest element of ${\mathcal O}(H)$, we have
\begin{equation}\label{e4.19}
{\rm Lin}({E})=\{x \in X \mid l_{\tilde{E}}(x)=0\,\,\text{for all}\,\,\tilde{E} \in {\mathcal O}(H)\,\,\text{such that}\,\,{E} \lhd_H^* \tilde{E}\},
\end{equation}
and ${\rm Lin}({E})= X$, when ${E}$ is the greatest element of ${\mathcal O}(H)$.
\end{theorem}

\begin{proof}
Suppose that ${E} \in {\mathcal O}(H)$ is not the greatest element of ${\mathcal O}(H)$.  Then the family $\{\tilde{E} \in {\mathcal O}(H) \mid {E} \lhd_H^* \tilde{E}\}$ is nonempty and we have, through the second assertion of Proposition \ref{pr4.11}, that $l_{\tilde{E}}(x)=0$ for all $x \in {\rm Lin}({E})$ whenever $\tilde{E} \in {\mathcal O}(H)$ is such that ${E} \lhd_H^* \tilde{E}$. Thus, the implication
$$
x \in {\rm Lin}({E}) \Longrightarrow l_{\tilde{E}}(x)=0\,\,\text{for all}\,\,\tilde{E} \in {\mathcal O}(H)\,\,\text{such that}\,\,{E} \lhd_H^* \tilde{E}
$$
is true.

Let us show that the reverse implication is also true. Suppose that for some $x \in X$ the equality $l_{\tilde{E}}(x) = 0$ holds for every $\tilde{E} \in {\mathcal O}(H)$ such that ${E} \lhd_H^* \tilde{E}$. Then it follows from the first equality from \eqref{e4.18} that $x \notin (-{\tilde{E}})\cup {\tilde{E}}$ for all $\tilde{E} \in {\mathcal O}(H)$ such that ${E} \lhd_H^* \tilde{E}$ and hence $x~\in~(-{E})\bigcup (-{\widehat{F}}({E}) \bigcup L_H \bigcup {\widehat{F}}({E}) \bigcup {E}  = {\rm Lin}({E}).$

It completes the proof of the first assertion.

The validity of the equality $X={\rm Lin}(E)$ for the greatest element $E \in ({\mathcal O}(H),\unlhd_H^*)$ follows from the equalities \eqref{e4.10a} and \eqref{e4.17a}.
\end{proof}

\begin{theorem}\label{th4.171}
Let $H \subset X$ be an asymmetric conical halfspace in $X$ and let ${\mathcal F}_H = \{l_E \mid E \in {\mathcal O}(H)\}$  be a family of linear functions associated with $H$. Then for each  $E \in {\mathcal O}(H)$ which is not the greatest element of $({\mathcal O}(H), \unlhd_H^*)$ the following equality
\begin{equation}\label{e4.21}
E= \{x \in X \mid l_{\tilde{E}}(x)=0\,\,\text{for all}\,\,\tilde{E} \in {\mathcal O}(H)\,\,
\text{such that}\,\,E \lhd_H^* \tilde{E}\,\,\text{and}\,\,l_E(x) > 0\}
\end{equation}
is true, while
\begin{equation}\label{e4.21a}
E= \{x \in X \mid l_E(x) > 0\}
\end{equation}
when $E \in {\mathcal O}(H)$ is the greatest element of $({\mathcal O}(H), \unlhd_H^*)$.

Furthermore, for every $E \in {\mathcal O}(H)$ the equality
\begin{equation}\label{e4.22}
L_E= \{x \in X \mid l_{\tilde{E}}(x)=0\,\,\text{for all}\,\,\tilde{E} \in {\mathcal O}(H)\,\,
\text{such that}\,\,E \unlhd_H^* \tilde{E}\}
\end{equation}
holds.

At last,
\begin{equation}\label{e4.23}
L_H = \{x \in X \mid l_{{E}}(x)=0\,\,\forall\,\,{E} \in {\mathcal O}(H)\}.
\end{equation}
\end{theorem}
\begin{proof}
The equalities \eqref{e4.21} and \eqref{e4.22} follows from \eqref{e4.18} and \eqref{e4.19}. Since ${\rm Lin}(E) = X$ when $E \in {\mathcal O}(H)$ is the greatest element of $({\mathcal O}(H), \unlhd_H^*)$ the equality \eqref{e4.21a} is actually the first equality from \eqref{e4.18}.

The equality \eqref{e4.23} follows from \eqref{e4.22} and the inclusion $L_H \subset L_E$ which holds for every $E \in  {\mathcal O}(H)$.
\end{proof}

Thus, we see from Theorems \ref{th4.161} and \ref{th4.171} that each open component $E$ of an asymmetric conical halfspace $H$  as well as the vector subspaces ${\rm Lin}(E)$ and $L_E$ related to $E$ can be completely characterized by linear functionals from the family ${\mathcal F}_H$. In the next theorem we highlight the main properties of the family ${\mathcal F}_H$.

\begin{theorem}\label{th4.18}
Let $H \subset X$ be an asymmetric convex halfspace in $X$ and let ${\mathcal F}_H = \{l_E \mid E \in {\mathcal O}(H)\}$  be a family of linear functions associated with $H$ and linearly ordered by the relation $\unlhd^*_H$.

For each $x \in X$ either the subfamily ${\mathcal F}_H(x):=\{l_E \in {\mathcal F}_H \mid l_E(x) \ne 0\}$ is empty or ${\mathcal F}_H(x)$ has the least (with respect to $\unlhd'_H$) element.

Furthermore, for each $l_E \in {\mathcal F}_H$ there exists $x \in X$ such that $l_E$ is the least (with respect to $\unlhd'_H$) element of ${\mathcal F}_H(x)$.
\end{theorem}

\begin{proof}
Since $L_H \subset L_E \subset {\rm Lin}(E)$ for all $E \in {\mathcal O}(H)$ (see Proposition \ref{pr4.2a}), then, through the second equality in \eqref{e4.18}, for every $E \in {\mathcal O}(H)$ we have $l_E(x)=0$ for all $x \in L_H$. Thus, ${\mathcal F}_H(x) = \varnothing$ for all $x \in L_H$.

Now, let $x \in X \setminus L_H$. Then, there exists $E \in {\mathcal O}(H)$ such that either $x \in E$ or $x \in -E$ and hence, due to the first equality in \eqref{e4.18}, we have either $l_E(x) > 0$  or $l_E(x) < 0$. Consequently, $l_E(x) \ne 0$  and ${\mathcal F}_H(x) \ne \varnothing$. Moreover, it follows from \eqref{e4.21} that $l_{\tilde{E}}(x)=0$ for all $l_{\tilde{E}} \lhd'_H l_E$. Thus, $l_E$ is the least element of ${\mathcal F}_H(x)$.

To prove the last assertion we note that each  $l_E \in {\mathcal F}_H$ is the least element of ${\mathcal F}_H(x)$ (with respect to $\unlhd'_H$) for $x \in E$.
\end{proof}

The properties of the family ${\mathcal F}_H$ described in Theorem \ref{th4.18} shows that ${\mathcal F}_H$ is a cortege of linear functions on $X$.

Recall the definition of a cortege of linear functions.

\begin{definition}{\rm \cite{Gor-m,GS2000,Gor21}}
{\rm A family of linear functions ${\mathcal F} \subset {\mathcal L}(X,{\mathbb{R}})$ is said to be a \textit{cortege of linear functions on $X$} if

$(i)$  ${\mathcal F}$ is ordered by some linear order $\unlhd_{{\mathcal F}}$;

$(ii)$ for each $x \in X$ either the subfamily ${\mathcal F}_x := \{l \in {\mathcal F} \mid l(x) \ne 0\}$ is empty or ${\mathcal F}_x$ has the least (with respect to $\unlhd_{{\mathcal F}}$) element $l_x$;

$(iii)$ for each $l \in {\mathcal F}$ there exists $x \in X$ such that  $l=l_x$, i.e., each $l \in {\mathcal F}$ is the least element ${\mathcal F}_x$ for some $x \in X$.}
\end{definition}
Families of linear functions satisfying conditions $(i)$ and $(ii)$ have been introduced by
Klee \cite{Klee} for analytic representation of semispaces. The term ``cortege of linear functions'' for families of linear function satisfying the conditions ($i$) -- ($iii$) has been introduced
by Gorokhovik \cite{Gor-m}. In \cite{GS2000} Gorokhovik and Shinkevich have introduced the concept of a cortege of affine functions generalizing the concept of a cortege of linear functions and used it for the analytical representation of arbitrary halfspaces.

Let $({\mathcal F},\unlhd_{\mathcal F})$ be a cortege of linear functions on $Y$.
For each $l \in {\mathcal F}$ we denote by $X_l$ the following vector subspace
$$
X_l := \{x \in X \mid l'(x)=0\,\,\text{for all}\,\,l' \lhd_{\mathcal F} l\}.
$$
If $l \in {\mathcal F}$ is the least element of ${\mathcal F}$ we suppose $X_l =X$.

It follows from $(ii)$ that if $x \in X$ is such that ${\mathcal F}_x \ne \varnothing$ then $x$ belongs $X_{l_x}$. It implies that $l_x(X_{l_x})= {\mathbb R}$  and, through $(iii)$, we conclude that every
$l \in {\mathcal F}$ is nonzero. Moreover (see \cite{GS2000}), any cortege $({\mathcal F}, \unlhd_{\mathcal F})$ is a linearly independent family in ${\mathcal L}(X,{\mathbb{R}})$.

Every cortege of linear functions generates a real-valued function called step-linear which is defined as follows.

\begin{definition}{\rm \cite{GS2000,Gor21}}
{\rm A real-valued function $u:X \to {\mathbb{R}}$ is called \textit{step-linear} if there exists a cortege of linear functions ${\mathcal F}$ on $X$ such that $u(x) = u_{{\mathcal F}}(x), x \in X$, where
$$
u_{\mathcal F}(x):=\left\{
  \begin{array}{cr}
       0,&\text{when}\,\,{{\mathcal F}}_{x} = \varnothing,\\
     l_x(x),&\text{when}\,\,{{\mathcal F}}_{x} \ne \varnothing.\\
  \end{array}
     \right.
$$
}
\end{definition}

Clearly, that any nonzero linear function $l$ can be considered as a step-linear one generated by the one-element cortege ${\mathcal F} =\{l\}$.

For a finite cortege of linear functions ${\mathcal F}:= \{l_1,l_2,\ldots,l_k\}$ which is ordered according to numbering ($l_1 \lhd_{\mathcal F} l_2 \lhd_{\mathcal F} \ldots \lhd_{\mathcal F} l_k$) the function $u_{\mathcal F}$ is defined as follows
\begin{equation*}
u_{{\mathcal F}}(x)=\left\{
  \begin{array}{ll}
       l_1(x),&\text{if}\,\,l_1(x) \ne 0,\\
       l_2(x),&\text{if}\,\,l_1(x) = 0,l_2(x) \ne 0,\\
       ......... & ....................................................................... \\
       l_{k-1}(x),&\text{if}\,\,l_1(x) = 0, \ldots,  l_{k-2}(x)=0,l_{k-1}(x) \ne 0,\\
       l_{k}(x),&\text{if}\,\,l_1(x) =0, \ldots, l_{k-1}(x)=0.
       \end{array}
     \right.
\end{equation*}

For the first time step-linear functions (under the name ``conditionally linear functions'') have been introduced in \cite{Gor-m}. Later step-linear functions as well as step-affine ones generalizing them have been studied in \cite{GS2000,Gor21}.

\begin{proposition}\label{pr5.5}
Every step-linear function $u:X \to {\mathbb{R}}$ is homogeneous, i.e., $u(tx)= tu(x)$ for all $x \in X$ and all $t \in {\mathbb{R}}$, and, moreover, for any $x,y \in X$ the following implications hold:
\begin{equation}\label{e5.22}
u(y)=0 \Longrightarrow u(x+y)=u(x)
\end{equation}
and
\begin{equation}\label{e5.23}
u(x)>0,u(y)>0 \Longrightarrow u(x+y)> 0.
\end{equation}
\end{proposition}

\begin{proof}
Let $u:x \to {\mathbb{R}}$ be a step-linear function and let ${\mathcal F}$ be a cortege of linear functions on $X$ such that $u(x) = u_{\mathcal F}(x)$ for all $x \in X$.

The homogeneity of $u$ follows from the homogeneity of the linear functions belonging to the cortege ${\mathcal F}$.

To prove the implications \eqref{e5.22} and \eqref{e5.23} we first observe that through additivity of linear functions  the equality
$l(x+y) = l(x)+l(y)$
holds for any $x,y \in X$ and for all $l \in {\mathcal F}$.

If $u(y)=0$ then $l(y) = 0$ for all $l \in {\mathcal F}$ and hence  $l(x+y) = l(x)$ for all $l \in {\mathcal F}$. This proves \eqref{e5.22}.

Assume that $u(x)>0,u(y)>0$. If $l_x \lhd_{\mathcal{F}} l_y$ then $l(x+y)=l(x)+l(y)=0$ for all $l \lhd_{\mathcal{F}} l_x$ and $l_x(x+y)=l_x(x)+l_x(y) = l_x(x)$. Thus, $l_{x+y}=l_x$ and hence $u(x+y)=u(x)>0$. Similarly, if $l_x \lhd_{\mathcal{F}} l_y$ we get that $u(x+y)=u(y)>0$. At last, if $l_x = l_y$, we have that $l(x+y)=l(x)+l(y)=0$ for all $l \lhd_{\mathcal{F}} l_x=l_y$ and $l_x(x+y)=l_x(x)+l_x(y)=l_x(x)+l_y(y)>0$. Thus, in this case we have that $l_{x+y}=l_x=l_y$ and $u(x+y)=u(x)+u(y)>0$. This completes the proof of the implication \eqref{e5.23}.
\end{proof}

\begin{theorem}\label{th5.7}
A subset $H \subset X$ is an asymmetric conical halfspace in $X$ if and only if there exists a step-linear function $u:X \to {\mathbb{R}}$ such that
\begin{equation}\label{e5.27}
x \in H \Longleftrightarrow u(x) > 0.
\end{equation}
Furthermore,
\begin{equation}\label{e5.28}
x \in L_H \Longleftrightarrow u(x) = 0
\end{equation}
for every step-linear function $u:X \to {\mathbb{R}}$ which satisfies \eqref{e5.27}.
\end{theorem}

\begin{proof}
It was shown above that each asymmetric conical halfspace $H \subset X$ can be associated with the family of linear functions ${\mathcal{F}}_H := \{l_E \in {\mathcal L}(X, \mathbb{R}) \mid E \in {\mathcal O}(H), l_E \,\,\text{satisfies}\,\,\eqref{e4.18}\}$. It follows from Theorem \ref{th4.18} that ${\mathcal{F}}_H$ ordered by the linear order $\unlhd'_H$ is a cortege of linear functions on $X$. Using the equality $H = \bigsqcup\{E \mid E \in {\mathcal O}(H)\}$ and the characterization of $E \in {\mathcal O}(H)$ given in Theorem \ref{th4.171} it is not difficult to verify that for the step-linear function $u_H : X \to {\mathbb{R}}$ generated by the cortege ${\mathcal{F}}_H$ we have $x \in H \Longleftrightarrow u_H(x) > 0$ and $x \in L_H \Longleftrightarrow u_H(x) = 0$.

Conversely, let $u:X \to {\mathbb{R}}$ be a step-linear function and let a subset $H \subset X$ be defined by \eqref{e5.27}. It follows from the property \eqref{e5.23} and from the homogeneity of $u$ that the subset $H$ is a convex cone. Furthermore, the homogeneity of $u$ implies that $H$ is asymmetric. Thus, the subset $H$ is an asymmetric convex cone. It is easy to see from \eqref{e5.27} that $u(x) =0$ for any $x \in X \setminus ((-H)\bigcup H)$. From this, since $u$ is homogeneous and satisfies the property \eqref{e5.23}, we deduce that $X \setminus ((-H)\bigcup H)$ is a vector subspace in $X$ and, consequently, $H$ is asymmetric conical halfspace with $L_H = X \setminus ((-H)\bigcup H) = \{x \in X \mid u(x) = 0\}$.
\end{proof}

The next theorem describes the class of such asymmetric conical halfspace for which in \eqref{e5.27} and \eqref{e5.28} step-linear functions can be replaced by linear functions.

\begin{theorem}\label{th5.8}
Let $H \subset X$ be an asymmetric conical halfspace in a real vector space $X$. Then the following statements are equivalent:

$($i$)$ $H$ is algebraic open, i.e., ${\rm cr}H \ne \varnothing$ and $H = {\rm cr}H$;

$($ii$)$ the family of open components of $H$ is a singleton, that is ${\mathcal O}(H) = \{H\}$;

$($iii$)$ there exists a linear function $l:X \to {\mathbb{R}}$ such that
\begin{equation}\label{e5.29}
x \in H \Longleftrightarrow l(x) > 0\,\,\text{and}\,\,x \in L_H \Longleftrightarrow l(x) =0.
\end{equation}
\end{theorem}

\begin{proof}
$(i) \Longleftrightarrow (ii)$. Let the statement $(i)$ be satisfied. Since ${\rm cr}H \ne \varnothing$, then ${\rm icr}H = {\rm cr}H = H$ and, hence, by \eqref{e2.22}  we obtain that $x \in H$ if and only if  $y \unlhd_H x\,\,\forall\,\,y \in H$. But the latter is equivalent to the assertion that $y \eq_H x$ for all $x,y \in H$ and hence ${\mathcal O}(H) = \{H\}$.

Conversely, let ${\mathcal O}(H) = \{H\}$. Since each open component of $H$ is relatively algebraic open then $H = {\rm icr}H$. But, $H$ is an asymmetric conical halfspace, therefore $X = (-H) \bigsqcup L_H \bigsqcup H$ and, consequently, ${\rm Lin}(H) = X$. It implies that ${\rm icr}H = {\rm cr}H = H$. Thus, $H$ is algebraic open.

The proof of the equivalence $(i) \Longleftrightarrow (ii)$ is complete.

$(i) \Longleftrightarrow (iii)$. It follows from $(i)$ and the equality $X = (-H) \bigsqcup L_H \bigsqcup H$ that $H$ is algebraic open halfspace in $X$ and  $L_H$ is a homogeneous hyperplane which is the boundary of $H$. From this we conclude that there exists a nonzero linear function $l: X \to {\mathbb{R}}$ such that $L_H =\{x \in X \mid l(x)=0\}$ and $H = \{x \in X \mid l(x)>0\}$.
This completes the proof of the implication $(i) \Longrightarrow (iii)$.

It is easily verified that the reverse implication $(iii) \Longrightarrow (i)$ is also true.
\end{proof}

\begin{remark}
{\rm Theorem \ref{th5.7} shows that step-linear functions can be considered as dual objects to conical halfspaces. It follows from Theorem \ref{th5.8} that this duality actually is an extension of the classical duality between algebraically    open halfspaces and linear functions.}
\end{remark}

\begin{example}
{\rm Consider the convex cone $H$ in ${\mathbb{R}}^3$ defined as follows: $H := \{(x_1,x_2,x_3) \in {\mathbb{R}}^3 \mid x_1>0\,\,\text{or}\,\,x_1=0,x_2>0\}$. It is easy to verify that $H$ is an asymmetric conical halfspace and $L_H = \{(x_1,x_2,x_3) \in {\mathbb{R}}^3 \mid  x_1=0, x_2=0\}$. For this asymmetric conical halfspace $H$ the step-linear function $u: {\mathbb{R}}^3 \to {\mathbb{R}}$ such that $u(x_1,x_2,x_3) = x_1$ when $x_1 \ne 0$ and $u(x_1,x_2,x_3) = x_2$ when $x_1 = 0, x_2 \ne 0$
satisfies \eqref{e5.27} and \eqref{e5.28}, while there is no linear function which satisfies \eqref{e5.29}. Note, that ${\rm cr}H = \{(x_1,x_2,x_3) \in {\mathbb{R}}^3 \mid x_1 > 0\} \ne H$ and, consequently, the statement $(i)$ from Theorem \ref{th5.8} does not hold.}
\end{example}

\section{Extension of asymmetric convex cones to asymmetric conical halfspaces and analytical representation of asymmetric convex cones}\label{sec6}

We say that an asymmetric convex cone $K_1 \subset X$ \textit{extends} another asymmetric convex cone $K_2 \subset X$ if $K_2\,\subseteq\,K_1$ and we say that an asymmetric convex cone $K_1$ \textit{regularly extends} an asymmetric convex cone $K_2$ if $K_2\,\subseteq\,K_1$ and, in addition, $L_{K_2}\,\subseteq\,L_{K_1}$.

\begin{example}
{\rm Consider convex cones  $K_1$ and $K_2$ defined on ${\mathbb R}^3$ as follows: $K_1 := (x_1,x_2,x_3) \mid (x_1 > 0, x_2 > 0)\,\,\text{or}\,\,(x_1 = 0, x_2 = 0, x_3 > 0)\}$ and $K_2 := \{(x_1,x_2,x_3) \mid x_1 > 0, x_2 > 0\}$. It is easy to see that $K_2 \subset K_1$ and hence $K_1$ extends $K_2$ but the extension $K_2$ to $K_1$ is not regular, since $L_{K_2} = \{(x_1,x_2,x_3) \mid x_1 = 0, x_2 = 0\} \not\subset L_{K_1} = \{(x_1,x_2,x_3) \mid x_1 = 0, x_2 = 0, x_3 = 0\}$.}
\end{example}

To prove the existence of regular extensions for each asymmetric convex cone we need the following theorem on separation of convex cones by conical halfspaces \cite{Gor21} which is presented here without proof.

\begin{theorem}{\rm \cite{Gor21}}\label{th6.1}
Let  $K_{1}$ and $K_{2}$ be convex cones in a real vector space $X$, and let $K_{1}$ be asymmetric. Then
$K_{1}\cap K_{2}=\varnothing$ if and only if there exists an asymmetric conical halfspace $H \subset X$ such that
$K_{1}\subset H$ and $K_2 \subset X \setminus H$ or,
 equivalently, if and only if there exists a step-linear function
$u : X \rightarrow {\mathbb R}$ such that
$$
u\,(x) > 0\,\,\text{for all}\,\,x\in K_{1},\quad
u\,(x)\leq 0\,\,\text{for all}\,\,x\in K_{2}.
$$
\end{theorem}

Note that the first part of Theorem \ref{th6.1} is the counterpart of the Kakutani-Tukey theorem on separation of convex sets by halfspaces (see \cite{Kakutani,Tukey} and \cite[Theorem 1.9.1]{HF}).

\begin{theorem}\label{th6.2}
Each asymmetric convex cone  $K\,\subset X$ can be regularly extended to an asymmetric conical halfspace $H \subset X$.
\end{theorem}

\begin{proof}
Let $K \subset X$ be an asymmetric convex cone in $X$ and let $L_K$ be the vector subspace associated with $K$. Since $K\bigcap L_K = \varnothing$, then, by Theorem \ref{th6.1} there exists an asymmetric conical halfspace $H \subset X$ such that $K \subset H$ and $L_K \subset X \setminus H$. It is easy to see that $H$ regularly extends $K$.
\end{proof}

\begin{theorem}\label{th6.3}
For any asymmetric convex cone  $K\,\subset X$ defined on a real vector space $X$ the family ${\mathcal U}_K$ of all step-linear functions $u: X \to {\mathbb{R}}$ such that
\begin{equation}\label{e6.1}
x \in K  \Longrightarrow u(x)>0,
\end{equation}
and
\begin{equation}\label{e6.2}
x \in L_K \Longrightarrow u(x)= 0.
\end{equation}
is nonempty and, moreover,
\begin{equation}\label{e6.3}
x \in K  \Longleftrightarrow u(x)>0\,\,\text{for all}\,\,u \in {\mathcal U}_K
\end{equation}
and
\begin{equation}\label{e6.4}
x \in L_K \Longleftrightarrow u(x)= 0\,\,\text{for all}\,\,u \in {\mathcal U}_K.
\end{equation}
\end{theorem}

\begin{proof}
By Theorem \ref{th6.2} any asymmetric convex cone  $K\,\subset X$ can be regularly extended to an asymmetric conical halfspace$ H \subset X$. Due to Theorem \ref{th5.7} for $H$ there exists a step-linear function $u_H:X \to {\mathbb R}$ satisfying \eqref{e5.27} and \eqref{e5.28}. Since $K \subset H$ and $L_K \subset L_H$, the step-linear function $u_H$ also satisfies \eqref{e6.1} and \eqref{e6.2} and, hence, the family ${\mathcal U}_K$ is nonempty.

Let us now prove \eqref{e6.3} and \eqref{e6.4}. The right implications ($\Longrightarrow$) in \eqref{e6.3} and \eqref{e6.4} follows from the definition of the family ${\mathcal U}_K$. To prove the left implication ($\Longleftarrow$) in \eqref{e6.3} it is sufficient to show that for any $y \not\in K$ there exists a step-linear function $\bar{u} \in {\mathcal U}_K$ for which the equality $\bar{u}(y) \leq 0$ holds. If $y \notin K$ then either $-y \in K$ or $y \in X \setminus ((-K)\bigcup K)$. In the case when $-y \in K$ we have, through \eqref{e6.1} and the homogeneity of step-linear functions, that $u(y) = -u(- y)<0$ for all $u \in {\mathcal U}_K$. Consider the case when $y \in X \setminus ((-K)\bigcup K)$. Then, either $y \in L_K$ or $y \in X \setminus ((-K)\bigcup L_K \bigcup K)$. If $y \in L_K$ through \eqref{e6.2} we have $u(y) = 0$ for all $u \in {\mathcal U}_K$. Thus, when $y \in -K$ or $y \in L_K$ we can take as $\bar{u}$ any step-linear function $u \in {\mathcal U}_K$. It remains to consider the case $y \in X \setminus ((-K)\bigcup L_K \bigcup K)$. Note, that $y \ne 0$ since $y \not\in L_K$. Let $l_y:=\{\alpha y \mid \alpha < 0\}$ be the ray emanating from the origin and going through $-y$. Prove that $({\rm conv}(K\bigcup l_y)) \bigcap L_K = \varnothing$, where ${\rm conv}M$ is the convex hull of the set $M$. Assume to the contrary that there exist $z \in K, \alpha < 0$ and $\lambda \in (0,1)$ such that $\lambda z + (1 - \lambda)\alpha y \in L_K$. Since $L_K$ is a vector subspace, $-z - \lambda^{-1}(1 - \lambda)\alpha y \in L_K$. By the definition of $L_K$ we obtain $z + (-z - \lambda^{-1}(1 - \lambda)\alpha y) = -\lambda^{-1}(1 - \lambda)\alpha y  \in K$, which contradicts the condition $y \not\in K$. The obtained contradiction proves that ${\rm conv}(K\bigcup l_y)) \bigcap L_K = \varnothing$. It is easy to verify that ${\rm conv}(K\bigcup l_y)$ is an asymmetric convex cone. Consequently, by Theorem \ref{th6.1} there exists a step-linear function $\bar{u}:X \to {\mathbb R}$ such that $\bar{u}(x) > 0$ for all $x \in {\rm conv}(K\bigcup l_y)$ and $\bar{u}(x) \leq 0$ for all $x \in L_K$. Since the function $\bar{u}$ is homogeneous and $-y \in {\rm conv}(K\bigcup l_{y})$ we have that $\bar{u}(y) < 0$ and, since $L_K$ is a vector subspace, we have $\bar{u}(x)=0$ for all $x \in L_K$. This completes the proof of \eqref{e6.3}.

To prove the left implication ($\Longleftarrow$) in \eqref{e6.4} we show that for any $y \not\in L_K$ there exists $u \in {\mathcal U}_K$ for which $u(y) \ne 0$. If $y \not\in L_K$ then the following three disjoint alternatives can be held: $y \in K,\,-y \in K$ and $y \in X \setminus ((-K)\bigcup L_K \bigcup K)$. For $y \in K$ and $-y \in K$ through \eqref{e6.1} we have $u(y) \ne 0$ for all $u \in {\mathcal U}_K$. While for the alternative $y \in X \setminus ((-K)\bigcup L_K \bigcup K)$, as it was shown above, there exists a step-linear function $u \in {\mathcal U}_K$ such that $u(y) < 0$. Thus, if $u(y)= 0\,\,\text{for all}\,\,u \in {\mathcal U}_K$ then $y \in L_K$.

This completes the proof of the theorem.
\end{proof}

\section{Conclusions}

In the literature the geometric structure of convex cones is most often associated with their facial structure.
In the paper we have developed another approach to studying the geometric structure of convex cones. In this approach so called open components rather than faces are considered as basic elements of geometric structure.  We refer to the geometric structure of convex cones based on open components as the internal geometric structure. We have focused on convex cones in infinite-dimensional real vector spaces which are endowed with no topology. It have been demonstrated that in the infinite-dimensional setting the internal geometric structure of a convex cone is related to its facial structure but differs from the latter.

The characterization of geometric structure of convex cones via their open components have allowed us to obtain, first, the exact analytical representation of conical halfspaces by step-linear functions (Theorem \ref{th5.7}) and, then, the exact analytical representation of any asymmetric convex cone by the family of step-linear functions (Theorem \ref{th6.3}). The correspondence between an asymmetric conical halfspace and a step-linear function representing it is actually the extension of the classical duality between algebraic open halfspaces and linear functions. The family of step-linear functions, which analytically describe an asymmetric convex cone, can  be considered as a counterpart of the positive dual cone.

Due to these results more thorough study of step-linear functions is of interest. In particular, it has been shown in Proposition~\ref{pr5.5} that every step-linear function is homogeneous and satisfies the implications \eqref{e5.22} and \eqref{e5.23}. The next question arises: is a homogeneous function satisfying the implications \eqref{e5.22} and \eqref{e5.23} step-linear? The positive answer to this question would allow us to give an axiomatic definition of a step-linear function. If the answer is negative, we will need to find out additional properties which are inherent to step-linear functions.

\smallskip
\smallskip

{\bf Acknowledgements}
The research was supported by the State Program for Fundamental Research of Republic of Belarus 'Convergence-2025', project 1.3.04.

\end{document}